\documentclass[11pt]{amsart}
\usepackage{amssymb}
\usepackage{amsfonts}
\usepackage{amsmath}
\usepackage{graphicx}
\usepackage{xcolor}
\usepackage{amsmath}
\usepackage{mathrsfs}
\usepackage{stmaryrd}
\usepackage{epsfig,color}
\usepackage{blindtext}
\usepackage{enumerate}
\usepackage{hyperref}
\usepackage{url}
\usepackage{bbm}
\usepackage{nicefrac,mathtools}
\usepackage{bm}   
\usepackage{tabularx}
\DeclareGraphicsExtensions{.pdf,.jpeg,.png}
\usepackage{epstopdf}
\usepackage{cancel} 
\usepackage[normalem]{ulem} 
\usepackage{verbatim} 
\usepackage{scalerel}
\usepackage{enumitem}
\newcounter{dummy}
\makeatletter
\newcommand\myitem[1][]{\item[#1]\refstepcounter{dummy}\def\@currentlabel{#1}}
\makeatother

\usepackage{subcaption}
\usepackage{soul}

\usepackage{color}
\usepackage[msc-links, lite]{amsrefs}
\usepackage{geometry}
\usepackage{stackengine,scalerel}
\usepackage{tikz-cd}

\newtheorem{theorem}{Theorem}[section]

\newtheorem{conjecture}[theorem]{Conjecture}
\newtheorem{proposition}[theorem]{Proposition}
\newtheorem{lemma}[theorem]{Lemma}
\newtheorem{corollary}[theorem]{Corollary}
\newtheorem{question}[theorem]{Question}

\newtheorem{claim}[]{Claim}
\newtheorem*{acknowledgements}{Acknowledgements}
\theoremstyle{definition}
\newtheorem{definition}[theorem]{Definition}
\theoremstyle{remark}
\newtheorem{remark}[theorem]{Remark}
\newtheorem{example}[]{Example}

\numberwithin{equation}{section}

\newcommand{\mf}{\mathbf}
\newcommand{\mb}{\mathbb}
\newcommand{\mc}{\mathcal}

\newcommand{\mk}{\mathfrak}

\newcommand{\wti}{\widetilde}

\newcommand{\Vol}{\mathrm{Vol}}
\newcommand{\Area}{\mathrm{Area}}
\newcommand{\Id}{\mathrm{Id}}

\newcommand{\inj}{\operatorname{inj}}
\newcommand{\codim}{\operatorname{codim}}
\newcommand{\bd}{\partial}

\newcommand{\rom}[1]{\expandafter\romannumeral #1}
\newcommand{\Rom}[1]{\uppercase\expandafter{\romannumeral #1}}

\DeclareMathOperator{\Ric}{Ric}

\DeclareMathOperator{\spt}{spt}
\DeclareMathOperator{\interior}{int}
\DeclareMathOperator{\closure}{Clos}
\DeclareMathOperator{\Clos}{Clos}

\DeclareMathOperator{\divergence}{div}

\newcommand{\red}{\color{red}}
\newcommand{\blue}{\color{blue}}

\title{Infinite existence of equivariant minimal hypersurfaces}

\author{Xingzhe Li}
\address{Cornell University, Department of Mathematics, Ithaca, New York 14850}
\email{xl833@cornell.edu}

\author{Tongrui Wang}
\address{School of Mathematical Sciences, Shanghai Jiao Tong University, 800 Dongchuan RD, Minhang District, Shanghai, 200240, China}
\email{wangtongrui@sjtu.edu.cn}

\begin{document}
\maketitle

\begin{abstract}
    For a closed Riemannian manifold $M$ with a compact Lie group $G$ acting by isometries, we show that there are infinitely many $G$-invariant minimal hypersurfaces. Under the assumption that $M$ contains at most a finite number of minimal $G$-hypersurfaces admitting no $G$-invariant unit normal, we further show that each $G$-homology class of $M$ admits infinitely many distinct realizations by embedded minimal $G$-hypersurfaces. The proof relies on a new algorithm that employs multi-stage maximal cuttings. As part of this work, we also established an equivariant min-max theory in manifolds with cylindrical ends.
\end{abstract}

\section{Introduction}

\subsection{Background and motivations.}


Minimal hypersurfaces arise as critical points of the area functional and enjoy nice variational properties. Their existence has been extensively studied for centuries, yet remains central to the field. A fundamental advancement comes from the works of Almgren \cite{almgren1962homotopy, almgren1965theory} in the 60s, where he initiated the min-max theory for the area functional. With subsequent refinements made by Pitts \cite{pitts2014existence} and the regularity theory of Schoen-Simon \cite{schoen1981regularity}, this min-max framework yields the existence of a smoothly embedded closed minimal hypersurface in a compact manifold $(M^{n + 1}, g)$ with $3 \leq n + 1 \leq 7$ (see \cite{smith1982minmax, colding2002min-max, de2013existence} for other variants of min-max theory). Inspired by the Morse theoretic heuristics, Yau proposed the following conjecture in his famous 1982 Problem Section \cite[Problem 88]{yau1982problem}.  



\begin{conjecture}[S.-T. Yau \cite{yau1982problem}]\label{Conj: Yau's Conjecture}
    Every closed three-dimensional manifold contains infinitely many (immersed) minimal surfaces.
\end{conjecture} 

For generic metrics, a stronger result establishing the dense distribution of minimal hypersurfaces
was obtained by Irie-Marques-Neves \cite{irie2018density} via a perturbation argument. Later, an alternative proof of Conjecture \ref{Conj: Yau's Conjecture} was provided by X. Zhou in \cite{zhou2020multiplicity}, where he resolved the \emph{Multiplicity One Conjecture} raised by Marques-Neves \cite{marques2021morse} in the Almgren-Pitts setting; see also the proof by Chodosh-Mantoulidis \cite{chodosh2020ACmultiplicity} using the Allen-Cahn approach in dimension three. For subsequent progress related to Conjecture \ref{Conj: Yau's Conjecture} in the generic case, we refer to \cite{marques2019equidistribution, liyy2023infinite, Song19, Song-Zhou21, LS23, li2024generic}, etc. 

In parallel with these developments, the conjecture for non-generic metrics was addressed with different strategies. When the manifold satisfies the ``Frankel property'', it was solved in dimension $3 \leq n + 1 \leq 7$ by Marques-Neves \cite{marques2017existence} using the asymptotic behavior of the volume spectrum from Gromov \cite{gromov1988width, gromov2003waists} and Guth \cite{guth2009minimax}. Built upon the work of Marques-Neves, A. Song \cite{song2023infinite} settled the conjecture in full generality by localizing the min-max constructions that appeared in \cite{marques2017existence} to some compact manifold with stable minimal boundary. In the non-generic case, some further works inspired by Song's approach include \cite{Song19, wangzc2020existence, sunao24}, etc.  

From Song's resolution of Conjecture \ref{Conj: Yau's Conjecture}, we see that there exists at least one homology class $\alpha\in H_n(M;\mb Z_2)$ that admits infinitely many distinct realizations by embedded minimal hypersurfaces. Throughout this paper, we use the term `homology' to refer to the homology with $\mb Z_2$-coefficients for simplicity. In particular, it follows from Zhou's resolution of the multiplicity one conjecture \cite{zhou2020multiplicity} that $\alpha=[0]$ is such a class under bumpy metrics or metrics of $\Ric_M>0$. When $H_n(M; \mathbb{Z}_2)$ has a nontrivial class, it is interesting to know whether the number of embedded minimal hypersurfaces that are not $2$-sided separating is infinite. We phrase this as the following conjecture.  

\begin{conjecture}\label{Conj: Infinitely many homologically nontrivial} 
Suppose $3 \leq n + 1 \leq 7$ and $H_n(M^{n+1};\mb Z_2)$ is nontrivial. Then there are infinitely many homologically nontrivial closed embedded minimal hypersurfaces. 
\end{conjecture}

Now we turn to the equivariant setting and consider analogous questions. The ambient space of interest will be a closed Riemannian manifold $M$ equipped with a compact Lie group $G$ acting by isometries satisfying that 
\begin{align*}
    3 \leq \codim(G\cdot x) \leq 7 \qquad \forall x\in M. \tag{$\star$}
\end{align*}
With this setup, we propose below an equivariant analogue of Yau’s conjecture.

\begin{conjecture}\label{Conj: Infinitely many G-invariant} 
    $M$ contains infinitely many closed embedded minimal $G$-hypersurfaces. 
\end{conjecture}

Here we summarize known results to Conjecture \ref{Conj: Infinitely many G-invariant}. Under the assumption that $\Ric_{M} > 0$, this is confirmed by the second author in \cite{wang2022min} 
using the \emph{equivariant Almgren-Pitts min-max theory} developed by him \cite{wang2022min, wang2023free, wang2023G-index}. For more results on equivariant min-max in various settings, we refer to  
\cite{pitts1987applications, pitts1988equivariant, ketover2016equivariant, liu2021existence, franz2021equivariant, wang2025density}. 
\medskip 

By further prescribing the $G$-homology class, we are led to the following stronger conjecture. Here and below, the prefix `$G$-' represents the $G$-invariance for short.  

\begin{conjecture}\label{Conj: Equivariant Conjecture} 
    $M$ contains infinitely many closed embedded minimal $G$-hypersurfaces $\Sigma$ in a given $G$-homology class, i.e. $\Sigma = \Sigma_0 + \partial \Omega$ as mod 2 cycles for a fixed $G$-hypersurface $\Sigma_0 \subset M$ and some open $G$-set $\Omega \subset M$.
\end{conjecture}

\begin{remark}
    Note that when $G = \{e\}$, the above conjecture reduces to the following: There are infinitely many closed embedded minimal hypersurfaces in a given homology class $[\Sigma_0] \in H_n(M^{n+1}; \mathbb{Z}_2)$.
\end{remark}



Recall that a $G$-invariant metric is called $G$-bumpy if every finite cover of a smooth embedded minimal $G$-hypersurface is non-degenerate. Recently, the second author \cite{wang2026homology} proved Conjecture \ref{Conj: Equivariant Conjecture} for $G$-bumpy metrics and for local $G$-boundaries $\Sigma_0$ by verifying a multiplicity one result in the equivariant min-max setting. The notion of local $G$-boundary is a natural generalization of $G$-boundaries and is defined in Section \ref{section: statement of the main results}. This, for instance, includes non-boundary-type hypersurfaces that appeared in \cite{lwy24} where X. Yao and the authors showed a multiple existence result for $1$-sided minimal surfaces with low genus in lens spaces. We also refer to \cite{wangzc2023FourSpheres} for the multiplicity one theorem in Simon-Smith min-max theory and its application to another Yau’s conjecture on minimal $2$-spheres in $S^3$.    

The main objective of the present work is to further investigate Conjectures \ref{Conj: Infinitely many homologically nontrivial},  \ref{Conj: Infinitely many G-invariant}, and  \ref{Conj: Equivariant Conjecture} in the general case where the metric may not be generic or with positive Ricci curvature.

\subsection{Statement of the main results.}\label{section: statement of the main results}  

In this paper, we will investigate how the intrinsic nature of the space of local $G$-boundaries $\mc {LB}^G(M;\mb Z_2)$ ensures the infinite existence of $G$-invariant minimal hypersurfaces. Recall that a closed $G$-invariant hypersurface $\Sigma_0$ is a \emph{local $G$-boundary} if for any $p \in M$, there is a small $G$-invariant neighborhood $U$ of $G \cdot p$ so that $\Sigma_0$ is the boundary of an open $G$-set $\Omega$ restricted in $U$. Two local $G$-boundaries $\Sigma_0$ and $\Sigma_0'$ will stay in the same $G$-homology class if they differ by an open $G$-set. Then we denote by $\mathcal{H}^G(M; \mathbb{Z}_2)$ the collection of all $G$-homology classes in $\mc {LB}^G(M;\mb Z_2)$, which carries a natural structure of a finite group. Given $\alpha\in \mc H^G(M;\mb Z_2)$, an embedded minimal $G$-hypersurface is said to be a {\em minimal realization} of $\alpha$ if it represents $\alpha$ as a local $G$-boundary. 

Given $\Sigma \subset M$ a closed $G$-connected $G$-hypersurface, we say $\Sigma$ is {\em $(G, 1)$-sided} if $\Sigma$ admits no $G$-invariant unit normal. 
This is a natural generalized notion since for trivial $G = \{e\}$, $G$-connected $(G, 1)$-sided hypersurfaces coincide with connected $1$-sided hypersurfaces. Our main result confirms Conjecture \ref{Conj: Equivariant Conjecture} provided that $M$ contains at most a finite number of $(G, 1)$-sided minimal hypersurfaces: 

\begin{theorem}\label{Thm: main}  
    Let $(M^{n+1},g_{_M})$ be a closed Riemannian manifold with a compact Lie group $G$ acting by isometries so that 
    \[3 \leq \codim(G \cdot p) \leq 7, \quad \forall p \in M. \tag{$\star$}\]
	Then 
	either
	\begin{itemize}
		\item[(i)] there are infinitely many closed embedded, $G$-connected, $(G,1)$-sided minimal hypersurfaces; or
		\item[(ii)] for any $\alpha\in \mc H^G(M;\mb Z_2)$, there are infinitely many minimal realizations of $\alpha$. 
	\end{itemize}
\end{theorem}

In particular, by forgetting about the $G$-homology class constraint, we readily obtain an affirmative answer to Conjecture \ref{Conj: Infinitely many G-invariant}:   

\begin{theorem}\label{Thm: infinitely many G-invariant} 
    Let $(M,g_{_M})$ and $G$ be as above. Then there are infinitely many closed embedded minimal $G$-hypersurfaces in $M$.  
\end{theorem}

The following result is a direct consequence of Theorem \ref{Thm: main}.   

\begin{corollary}\label{main cor: H^G=Z_2}
    Let $(M, g_{_M})$ and $G$ be as above. Suppose that the only nontrivial element of $\mathcal{H}^G(M;\mb Z_2)$ is $\alpha$. Then there are infinitely many minimal realizations of $\alpha$. 
\end{corollary}

The proof of Theorem \ref{Thm: main} builds on an equivariant version of the approach in \cite{song2023infinite}, but there are several challenges caused by our tighter constraints. We will describe those challenges and our new ideas as follows. 

In higher-dimensional closed manifolds ($n \geq 7)$, the minimal hypersurfaces produced by Almgren-Pitts min-max theory might admit a singular set of codimension at least $7$, and Song's cylindrical construction \cite{song2023infinite}*{Section 2.2} cannot be adapted to prove the infinite existence result. Fortunately, for $(M, g_M)$ with a $G$-action satisfying ($\star$), the minimal $G$-hypersurfaces produced by the equivariant Almgren-Pitts min-max theory are smooth. This motivates us to develop an equivariant version of Song's cylindrical construction (Section \ref{section: construction of approximations}) and establish an equivariant min-max theory in manifolds with cylindrical ends (Section 
\ref{section: equivariant min-max theory in manifolds with cylindrical ends}). In particular, we prove in Theorem \ref{Thm: width estimates in cylindrical mfd} the equivariant cylindrical Weyl law and prove in Theorem \ref{Thm: confined min-max hypersurface} the equivariant cylindrical min-max theorem with sweepouts in prescribed homology classes.       

Another challenge comes from the cutting procedure. Recall that in \cite{song2023infinite}, a key ingredient is to cut along all minimal hypersurfaces with non-expanding neighborhoods, thereby reducing the ambient manifold to a core manifold with boundary satisfying the Frankel property. In our setting, one major difference is that we only assume the minimal realizations of $\alpha$ are finite for a contradiction. This raises the following issues: 
\begin{enumerate} 
    \item whether the cutting terminates in finite time; 
    \item whether the cutting relates the homology class in the core to that in the original manifold.  
\end{enumerate}

To overcome these, we introduce a new algorithm that employs multi-stage ``maximal'' cuttings. The homological information is carefully recorded at each stage to ensure both (1) and (2). We refer to Section \ref{section: Overview of proof} for an outline of this procedure and how this leads to a proof of Theorem \ref{Thm: main}. 
\medskip 

We now list several classes of manifolds with isometric Lie group actions satisfying ($\star$) so that our main theorems apply.  


\begin{example}
    Let $N^p$ be a $p$-dimensional closed Riemannian manifold with $3 \leq p \leq 7$ and let $S^q$ be a $q$-dimensional round sphere. Consider $M = N^p \times S^q$ endowed with the product metric and $G = SO(q + 1)$ acting on 
    $N^p$ trivially and on $S^q$ by rotation. By Theorem \ref{Thm: infinitely many G-invariant}, we conclude that there are infinitely many $SO(q + 1)$-invariant minimal hypersurfaces in $N^p \times S^q$.   
\end{example}

\begin{example}
    Let $N^p$ be as above and let $E^{p+q} \rightarrow N^p$ be a nontrivial principal $G$-bundle. Consider $M := E$ endowed with the connection metric and $G$ acting fiberwise. By Theorem \ref{Thm: infinitely many G-invariant}, we conclude that there are infinitely many $G$-invariant minimal hypersurfaces in $E$.     
\end{example}

\begin{example}
    Let $N^p$ be as above and let $E \rightarrow N^p$ be a nontrivial real vector bundle of rank $(q + 1)$ over $N$. Consider $M := S(E)$ the sphere bundle over $N^p$ endowed with the Sasaki metric and $G = SO(q + 1)$ acting fiberwise on $S^q$ by rotation. By Theorem \ref{Thm: infinitely many G-invariant}, we conclude that there are infinitely many $SO(q + 1)$-invariant minimal hypersurfaces in $S(E)$.     
\end{example}


\begin{example}
    Let $S^4$ be parameterized by $\{(x,z_1,z_2)\in \mb R\times \mb C^2: x^2+|z_1|^2+|z_2|^2=1\}$, and $\lambda\in S^1$ act on $S^4$ by the diffeomorphism $\lambda\cdot (x,z_1,z_2):= (x, \lambda^pz_1, \lambda^qz_2)$, where $p,q\in\mb N$ with $(p,q)=1$. 
    Then for any $S^1$-invariant Riemannian metric on $S^4$, there exist infinitely many $S^1$-invariant minimal hypersurfaces. 
    In particular, if $p=q=1$, then $S^4/S^1$ is homeomorphic to an antipodal symmetric $S^3$ with two singular points so that the $S^1$-actions can be extended to the $G=(S^1\rtimes \mb Z_4)/\mb Z_2$ actions on $S^4$ with $S^4/G$ homeomorphic to $RP^3$ (cf. \cite{wwz26}). 
    By Corollary \ref{main cor: H^G=Z_2}, for any $G$-invariant metric on $S^4$, there are infinitely many embedded closed minimal $G$-hypersurfaces $\Sigma$ in $S^4$ so that $\Sigma/G$ is a non-orientable surface homologous to an $RP^2$ in $S^4/G$. 
\end{example}

As a special case of Theorem \ref{Thm: main}
when $G = \{e\}$, we have the following corollary.  

\begin{corollary}\label{Coro: main} 
	Let $(M^{n+1},g_{_M})$ be a closed Riemannian manifold of dimension $3 \leq n + 1 \leq 7$.  
	Then 
	either
	\begin{itemize}
		\item[(i)] there are infinitely many closed embedded, connected, $1$-sided minimal hypersurfaces; or
		\item[(ii)] for any $\alpha\in H_n(M;\mb Z_2)$, there are infinitely many minimal realizations of $\alpha$. 
	\end{itemize}
\end{corollary}

In particular, by forgetting about the homology class constraint, we readily obtain an affirmative answer to Conjecture \ref{Conj: Infinitely many homologically nontrivial}.  

\begin{corollary}\label{cor: confirm homologically nontrivial}
    Let $(M^{n+1},g_{_M})$ be a closed Riemannian manifold of dimension $3 \leq n + 1 \leq 7$ with nontrivial $H_{n}(M; \mathbb{Z}_2)$. Then there are infinitely many homologically nontrivial closed embedded minimal hypersurfaces.   
\end{corollary}

\subsection{Further questions}

Let us first restrict to the case when $G = \{e\}$. 
Note that given $\alpha \in H_n(M; \mathbb{Z}_2)$, our Definition \ref{Def: minimal realization} allows the minimal hypersurfaces in $\alpha$ to be disconnected. One may wonder what would happen if the connectedness is further imposed. This gives rise to the following question.  

\begin{question}\label{question: connected minimal surface} 
    Suppose $H_n(M^{n + 1}; \mathbb{Z}_2)$ is nontrivial. 
    Are there infinitely many closed embedded, connected, minimal hypersurfaces in a given homology class $\alpha \in H_n(M; \mathbb{Z}_2)$? 
\end{question} 

Even in dimension three, the question above is already quite interesting. For instance, one may consider the lens space $L(p, q)$ where $p$ is an even integer and $q \in [1, p)$ is an integer coprime to $p$. In the case that $0 \neq \alpha \in H_2(L(p, q); \mathbb{Z}_2) \cong \mathbb{Z}_2$, Question \ref{question: connected minimal surface} reduces to:  

\begin{question}\label{question: connected minimal surface for Lens spaces} 
    Given any Riemannian $L(p, q)$ with $p, q$ above, are there infinitely many connected $1$-sided embedded minimal surfaces?     
\end{question}


Note that Question \ref{question: connected minimal surface for Lens spaces} has been confirmed for metrics with positive Ricci curvature 
in \cite{wang2026homology} (see Theorem 1.4). For general metrics, our result asserts that either $\alpha \in H_2(L(p, q); \mathbb{Z}_2)$ admits infinitely many minimal realizations or there are infinitely many embedded, connected, $1$-sided minimal surfaces.  
\medskip

Likewise, in the case that $G$ is nontrivial, it is natural to ask the following question. 

\begin{question}
    Suppose $\mathcal{H}^{G}(M^{n + 1}; \mathbb{Z}_2)$ is nontrivial. Are there infinitely many closed embedded, $G$-connected, minimal $G$-hypersurfaces in a given $G$-homology class $\alpha \in \mathcal{H}^G(M; \mathbb{Z}_2)$? 
\end{question}

\subsection{Overview of proof.}\label{section: Overview of proof}  
For simplicity, let us focus on the proof of the non-equivariant version, namely the proof of Corollary \ref{Coro: main}. Assuming that (i) does not hold, we aim at showing that $\alpha \in H_n(M; \mathbb{Z}_2)$ admits infinitely many minimal realizations. Suppose by contradiction that there are finitely many such realizations.    

For the easier case $\alpha = [0]$, we 
adapt the cutting procedure by Song, i.e. for $i = 0, 1, \ldots$, we cut along a connected minimal hypersurface in the interior of $N_i\subset M$ with non-expanding neighborhood to get a new domain $N_{i + 1} \subset M$ whose boundary (may be empty) has a contracting neighborhood. A key observation is that by our contradiction assumption and the fact that $H_n(M; \mathbb{Z}_2)$ is finite, the maximal cuttings must terminate in finite times. This ends up giving a core $N_m$ whose interior contains finitely many $1$-sided or $2$-sided separating minimal hypersurfaces, each with an expanding neighborhood, which automatically satisfies the Frankel property for minimal hypersurfaces. If $N_m$ has no boundary i.e. $N_m = M$, we derive a contradiction by combining the proof of  \cite{marques2017existence}*{Theorem 6.1 and 1.1} with the non-linear growth rate of the equivariant volume spectrum (\cite{wang2025density}*{Theorem 1.5}). Otherwise, the contradiction follows from combining Song's arguments \cite{song2023infinite} with the equivariant cylindrical Weyl law \eqref{Eq: Weyl law} and the associated min-max theorem (Theorem \ref{Thm: confined min-max hypersurface}(i)).



Now we turn to the case $\alpha \neq [0]$. As we only assume finitely many minimal realizations of $\alpha$, it is crucial to keep track of the homology classes both in $M$ and in each domain $N_i$. Our new algorithm employs multi-stage ``maximal'' cuttings, with the homological information recorded at each stage. To be more specific, we first adapt the cutting procedure to the essential area minimizer $\Sigma$ of $\alpha$. \footnote{\,Here essential means that no proper union of components of $\Sigma$ lies in $[0]$.} Denote by $\Sigma_0$ the union of all its components with non-expanding neighborhoods. After cutting along $\Sigma_0$, we obtain a domain $N_0 \subset M$ whose boundary has a contracting neighborhood. We record $\alpha_0 = \alpha - [\Sigma_0]$ as the remaining homology class in $M$ and $\wti{\alpha}_0$ as the relative homology class in $N_0$. Note that $\wti{\alpha}_0$ admits finitely many closed minimal realizations in the interior of $N_0$: indeed any such $\Gamma$ can be paired with $\Sigma_0\triangle\Gamma' := (\Sigma_0 \setminus \Gamma') \cup (\Gamma' \setminus \Sigma_0)$ for a union of some components $\Gamma'$ of $\partial N_0$ to get a minimal realization of $\alpha$. If $\alpha_0 = [0]$, then we are back to the easier case and achieve a contradiction.         
 
Suppose $\alpha_0 \neq [0]$. We proceed to adapt the cutting procedure to all closed non-expanding minimal realizations of $\wti{\alpha}_0$ in $\interior(N_0)$. Specifically, for any such realization $\Gamma$, our cutting procedure provides a new intermediate domain $N_1'\subset N_0\subset M$ with a similar boundary property to $N_0$ so that any connected component of $\Gamma$ either lies in $\interior(N_1')$ with an expanding neighborhood or lies entirely in $N_0\setminus\interior(N_1')$. This allows us to record the portion of $\Gamma$ remaining within $\interior(N_1')$ as the updated homology class $\alpha_1'$ in $M$ and the relative class $\wti\alpha_1'$ in $N_1'$.
Again by our contradiction assumption and the fact that $H_n(M; \mathbb{Z}_2)$ is finite, the maximal cuttings must terminate in finite times: either we exhaust the homology classes of $H_n(M; \mathbb{Z}_2)$ and return to the easier case, or we exhaust minimal realizations with non-expanding neighborhoods. Assuming that the latter case occurs, this gives us a domain $N_1 \subset M$ whose boundary (may be empty) has a contracting neighborhood. 
We record $\alpha_1$ as the remaining homology class in $M$ and $\wti{\alpha}_1$ as the relative homology class in $N_1$. By the same reasoning as above, $\wti{\alpha}_1$ admits finitely many closed expanding minimal realizations in the interior of $N_1$.      

To turn $N_1$ into a core, we need to further cut along closed minimal hypersurfaces with non-expanding neighborhoods that are disjoint from an expanding minimal realization of $\wti{\alpha}_1$. Like before, the maximal cuttings must terminate in finite times. Assuming that we have exhausted minimal hypersurfaces with non-expanding neighborhoods, this ends up giving us a core $N_2$ whose boundary (may be empty) has a contracting neighborhood. We record $\alpha_2$ as the remaining homology class in $M$ and $\wti{\alpha}_2$ as the relative homology class in $N_2$. By construction, $\wti{\alpha}_2$ admits finitely many closed expanding minimal realizations in the interior of $N_2$. Moreover, any such minimal realization of $\wti \alpha_2$ is $1$-sided or $2$-sided separating in $N_2$, and satisfies a Frankel-type property: it must intersect any closed minimal hypersurfaces in the interior of $N_2$. If $N_2$ has no boundary i.e. $N_2 = M$, then $\wti \alpha_2 = \alpha \neq [0]$ and every min-max width $\omega_p(\wti{\alpha}_2)$ is achieved by the area of a connected closed minimal realization of $\wti\alpha_2$ in the interior of $N_2$ with integer multiplicity. Then we derive a contradiction by combining the proof of  \cite{marques2017existence}*{Theorem 6.1 and 1.1} with the non-linear growth rate of the equivariant volume spectrum (\cite{wang2025density}*{Theorem 1.5}). In the case that $N_2$ has a boundary, if a closed minimal realization of $\wti \alpha_2$ is $2$-sided separating in $N_2$, then we may argue by a pairing construction as before and get infinitely many minimal realizations of $\alpha$, which gives a contradiction. Hence, any closed minimal realization of $\wti \alpha_2$ in the interior of $N_2$ is $1$-sided, which implies that $\wti \alpha_2 \neq 0$. Then the contradiction follows from combining Song's arguments \cite{song2023infinite} with the equivariant cylindrical Weyl law \eqref{Eq: Weyl law} and the associated min-max theorem with sweepouts in a prescribed {non-trivial} homology class (Theorem \ref{Thm: confined min-max hypersurface}(ii)). 




\begin{acknowledgements}
	Part of this work was done when T.W. visited Professor Xin Zhou at Cornell University; he would like to thank them for their hospitality. 
    T.W. is supported by the National Natural Science Foundation of China 12501076, 
    and the Natural Science Foundation of Shanghai 25ZR1402252. 

    X.L. would like to thank his advisor Xin Zhou for his encouragement. X.L. is partially supported by Simons Dissertation Fellowship and NSF grant DMS-2243149.  
\end{acknowledgements}

\section{Preliminaries}

Let $(M^{n+1}, g_{_M})$ be a connected closed $(n+1)$-dimensional Riemannian manifold, and $G$ be a compact Lie group acting by isometries on $M$ so that for all $p\in M$,
\begin{align}\label{Eq: dimension assumption}
	3\leq \codim(G\cdot p)\leq 7.
\end{align} 
Let $l+1$ be the cohomogeneity of the $G$-action on $M$ defined by:
\begin{align}\label{Eq: cohomogeneity}
	l+1:=\min \{ \codim(G\cdot p) : p\in M \}. 
\end{align}
Without loss of generality, $M$ is assumed to be orientable (see \cite{bredon1972introduction}*{Chapter I, Corollary 9.4}) and the action of $G$ is assumed to be effective. 
Define $\mu$ to be the normalized bi-invariant Haar measure on $G$ with $\mu(G)=1$, and let $\pi: M\to M/G$ be the natural quotient map.

\subsection{Notations for manifolds and group actions}

In this subsection, we introduce some notations for the Lie group actions (c.f.                                              \cite{bredon1972introduction, wall2016differential}). 

Firstly, by \cite{moore1980equivariant}, there is an orthogonal representation $\rho:G\to O(L)$ and an isometric embedding $i:M \to \mb R^L$ so that $i(g\cdot x) = \rho(g)\cdot i(x) $. 
For simplicity, denote by $g\cdot p$ the acting of $g\in G$ on $p\in\mb R^L$. 
Given a subset (resp. submanifold, hypersurface,...) $A \subset M$, $A$ a said to be a $G$-set (resp. $G$-submanifold, $G$-hypersurface,...) if $G \cdot A = A$.  
We also use the following definitions.
\begin{definition}
	Let $U,V\subset M$ be $G$-invariant subsets. 
	Then $U$ is said to be {\em $G$-connected} if for any connected components $U_i,U_j$ of $U$, there exists $g_{i,j}\in G$ with $g_{i,j}\cdot U_j=U_i$. 
	Additionally, we say $U$ is a {\em $G$-component} of $V$, if $U$ is a $G$-connected union of components of $V$. 
\end{definition}

Given $p\in M\subset \mb R^L$, an open $G$-set $U\subset M$, and $r,s,t >0$, we use the following notations:
\begin{itemize}
    \item $B_r(p), \mb B^L_r(p)$: the geodesic $r$-ball in $M$ and the Euclidean $r$-ball in $\mb R^L$ respectively;
    \item $B_r(G\cdot p)$: the geodesic $r$-tube around $G\cdot p$ in $M$;
    \item $A_{s,t}(G\cdot p)$: the open annulus $B_t(G\cdot p)\setminus B_s(G\cdot p)$;
    \item $\closure(A)$: the closure of a set $A\subset \mb R^L$; 
    \item $T_q(G\cdot p)$: the tangent space of the orbit $G \cdot p$ at $q \in G \cdot p$;
    \item $N_q(G \cdot p)$: the normal vector space of the orbit $G \cdot p$ in $M$ at $q \in G \cdot p$;
    \item $N(G \cdot p)$: the normal vector bundle of the orbit $G \cdot p$ in $M$;
    \item $\exp_p,\exp^\perp_{G\cdot p}$: the exponential map in $M$ at $p$ and the normal exponential map at $G\cdot p$; 
    \item $\inj(p),\inj(G\cdot p)$: the injectivity radius of $\exp_p$ and $\exp^\perp_{G\cdot p}$ respectively;
    \item $\exp_\Sigma^\perp$: the normal exponential map at a hypersurface $\Sigma$;
    \item $\mk X(M), \mk X(U)$: the space smooth vector fields supported in $M$ or $U$;
    \item $\mk X^G(M):=\{X\in \mk X(M): dg(X)=X\mbox{ for all $g\in G$} \}$, and $\mk X^G(U):=\mk X(U)\cap \mk X^G(M)$. 
\end{itemize}
In certain contexts, we would also consider 
\begin{itemize}
    \item $G_\pm\subset G$: $G_+$ is an index $2$ Lie subgroup of $G$, and $G_-:=G\setminus G_+$; 
    \item $C^k_G(M):=\{f\in C^k(M): f(g\cdot p)=f(p)\mbox{ for all $g\in G$}\}$ the space of $G$-invariant functions;
    \item $C_{G_{ \pm}}^{k}(M):=\left\{f \in C^{k}(M): \begin{array}{ll}f(p)=f(g \cdot p), & \forall p \in M, g \in G_{+} \\ f(p)=-f(g \cdot p), & \forall p \in M, g \in G_{-}\end{array}\right\}$ the space of {\em $G_\pm$-signed-symmetric} functions,
\end{itemize}
where $k\in\mb N\cup\{\infty\}$. 

For $p\in M$, let $G_p := \{g\in G: g\cdot p = p\}$ be the {\em isotropy group} of $p$ in $G$. 
Then we say $G\cdot p$ has the {\em orbit type} $(G_p)$, where $(G_p)=\{g\cdot G_p\cdot g^{-1}:g\in G\}$ is the conjugate class of $G_p$ in $G$. 
Denote by $M_{(G_p)}:= \{q\in M: (G_q)=(G_p)\}$ the $(G_p)$ orbit type stratum. 
By \cite{bredon1972introduction}*{Chapter 6, Corollary 2.5}, $M_{(G_p)}$ is a disjoint union of smooth embedded submanifolds of $M$. 
Moreover, there is a (unique) minimal conjugate class of isotropy groups $(P)$, known as the {\em principal orbit type}, so that $M^{prin}:=M_{(P)}$ is an open dense submanifold of $M$.

\subsection{Notations for geometric measure theory}

We now collect some notations in geometric measure theory (c.f. \cite{simon1983lectures, pitts2014existence}*{\S 2.1}). 
Let $\mc H^k$ be the $k$-dimensional Hausdorff measure in $\mb R^L$. 
Then denote by $\mf I_k(M; \mb Z_2)$ the space of $k$-dimensional mod $2$ flat chains in $\mb R^L$ supported in $M$; $\mc C(M)$ (resp. $\mc C(U)$) the space of all Caccioppoli sets in $ M$ (resp. in $U\subset M$); $\mc V_k(M)$ the closure (in the weak topology) of the space of $k$-dimensional rectifiable varifolds in $\mb R^L$ supported in $M$; $\mc F$ and $\mf M$ the flat (semi-)norm and the mass norm on $\mf I_k(M;\mb Z_2)$ respectively; $\mf F$ the varifolds $\mf F$-metric on $\mc V_k(M)$ and the currents $\mf F$-metric on $\mf I_k(M;\mb Z_2)$ (\cite{pitts2014existence}*{2.1(19)(20)}).   
Denote by  
\[\mc Z_n(M; \mb Z_2):=\{T\in \mf I_n(M;\mb Z_2): \bd T=0\}. \]
Let $\llbracket \Sigma \rrbracket$ be the mod $2$ flat chain associated with a submanifold $\Sigma\subset M$, and $\bd\Omega$ be the element in $\mc Z_n(M; \mb Z_2)$ reduced by the boundary of $\Omega\in \mc C(M)$. 
Then, define
\[\mc B(M;\mb Z_2):=\{T\in\mc Z_n(M;\mb Z_2): T=\bd \Omega \mbox{ for some $\Omega\in \mc C(M)$}\}.\]

Given $T\in \mf I_k(M;\mb Z_2)$ and $V\in \mc V_k(M)$, let $|T|$ be the integer rectifiable varifold induced by $T$; let $\|T\|$ and $\|V\|$ be the Radon measure induced by $|T|$ and $V$ respectively. 
In addition, if $g_\# T=T$ (resp. $g_\#V=V$) for all $g\in G$, then we say $T$ (resp. $V$) is {\em $G$-invariant}. 
We then define the following subspaces of $G$-invariant objects with naturally induced metrics $\mc F,\mf F, \mf M$. 
\begin{itemize}
    \item $\mf I_k^G(M;\mb Z_2):=\{T\in \mf I_k(M;\mb Z_2): g_\#T=T,~\forall g\in G\}$;
    \item $\mc C^G(M):=\{\Omega\in\mc C(M): g\cdot \Omega=\Omega, ~\forall g\in G\}$, and $\mc C^G(U):=\mc C^G(M)\cap\mc C(U)$;
    \item $\mc Z_n^G(M;\mb Z_2):=\{T\in \mc Z_n(M;\mb Z_2): g_\#T=T,~\forall g\in G\}$;
    \item $\mc B^G(M;\mb Z_2):=\{T\in \mc B(M;\mb Z_2): T=\bd\Omega, {\rm ~for~some~}\Omega\in \mc C^G(M)\}$;
    \item $\mc V^G_n(M):=\{V\in\mc V_n(M): g_\#V=V,~\forall g\in G\}$.
\end{itemize}
We call the element in $\mc B^G(M;\mb Z_2)$ a {\em $G$-boundary}.
Moreover, we also consider the following space of {\em local $G$-boundaries}: 
\begin{align}\label{Eq: local boundary}
        \mc {LB}^G(M;\mb Z_2):=
        \left\{T\in \mc Z_n^G(M;\mb Z_2): 
        \begin{array}{ll}
        		\mbox{for all } p\in M, \mbox{ there exists } r\in (\inj(G\cdot p)/2,\inj(G\cdot p)) \\
        		\mbox{ and } \Omega\in \mc C^G(B_r(G\cdot p))
        		\mbox{ so that } T=\bd\Omega \mbox{ in } B_r(G\cdot p)
        \end{array}
        \right\},
\end{align}
Note that $\mc {LB}^{\{\Id\}}(M;\mb Z_2) = \mc Z_n(M;\mb Z_2)$ if $G=\{\Id\}$. 

Suppose $G_+<G$ is an index $2$ Lie subgroup of $G$ and $G_-=G\setminus G_+$. 
Then we say $\Omega\in\mc C(M)$ and $T=\bd\Omega \in \mc B(M;\mb Z_2)$ are {\em $G_\pm$-signed-symmetric} if $G_+\cdot \Omega=\Omega$ and $G_-\cdot\Omega=M\setminus \Omega$ in $\mc C(M)$. 
Denote by 
\begin{itemize}
	\item $\mc C^{G_\pm}(M):=\{\Omega\in \mc C(M): \mbox{$\Omega$ is $G_\pm$-signed-symmetric}\}$;
	\item $\mc B^{G_\pm}(M;\mb Z_2):=\{T\in\mc Z^G_n(M;\mb Z_2): T=\bd\Omega \mbox{ for some }\Omega\in\mc C^{G_\pm}(M)\}$.
\end{itemize}

\subsection{Equivariant hypersurfaces}

In a closed manifold $M^{n+1}$, it is well known that every (mod $2$) homology class $\alpha\in H_n(M;\mb Z_2)$ can be {\em realized by a hypersurface $\Sigma\subset M$}, i.e. $i_*[\Sigma]=\alpha$ for the embedding $i:\Sigma\to M$ (see \cite{Thom}*{Theorem II.26}). 
In addition, since $\pi_1(M)$ is finitely generated, we also know that $H_n(M;\mb Z_2)$ is a finite group by Poincar\'e  duality. 
Hence, it follows from Song's resolution of Yau's Conjecture \cite{song2023infinite} that there is at least one homology class $\alpha\in H_n(M;\mb Z_2)$, $3 \leq n + 1 \leq 7$, with infinitely many distinct realizations by embedded minimal hypersurfaces. 

In the equivariant case, the equivariant (co)homology was introduced by Borel and also by Bredon \cite{bredon1967equivariant}. 
However, it is not clear (to our best knowledge) whether the realization by $G$-submanifolds holds in these equivariant homology constructions. 
Nevertheless, we only need the following definitions to derive our results. 

\begin{definition}\label{Def: equivariant homology}
	Given $T\in \mc Z_n^G(M;\mb Z_2)$ we define the {\em $G$-homology class} of $T$ by   
	\[[T]^G:=\{S\in \mc Z_n^G(M;\mb Z_2) : S=  T +\bd\Omega, \mbox{ for some } \Omega \in \mc C^G(M)\}. \]
	Denote by $\wti {\mc H}^G(M;\mb Z_2)$ and $\mc H^G(M;\mb Z_2)$ the collections of all $G$-homology classes in $\mc Z_n^G(M;\mb Z_2)$ and $ \mc {LB}^G(M;\mb Z_2)$ respectively. 
\end{definition}

\begin{definition}\label{Def: minimal realization}
	Given $\alpha\in \wti{\mc H}^G(M;\mb Z_2)$, an embedded (resp. minimal) $G$-hypersurface $\Sigma\subset M$ is said to be a {\em realization} (resp. {\em minimal realization}) of $\alpha$ if $\llbracket \Sigma \rrbracket\in\alpha$. 
	
	Moreover, let $\{\Sigma_i\}_{i=1}^m$ be the set of $G$-components of $\Sigma$. Then the realization $\Sigma$ of $\alpha$ is said to be {\em essential} if $\llbracket \cup_{i\in \mc I}\Sigma_i \rrbracket \notin [0]^G$ for any non-empty proper subset $\mc I\subsetneq \{1,\dots,m\}$.  
\end{definition}

Since our variational constructions are mainly applied in $\mc {LB}^G(M;\mb Z_2)$, we would mainly consider $\mc H^G(M;\mb Z_2)$ in the following. 

Clearly, $\bd B_r(G\cdot p)$, with small $r>0$, is an essential realization of $[0]^G$. Indeed, the following lemma indicates that every $\alpha\in\mc H^G(M;\mb Z_2)$ has an essential realization, which can be further taken to be minimal and area minimizing if $\alpha\neq [0]^G$. 

\begin{lemma}\label{Lem: minimizer representative}
	For any $\alpha\neq [0]^G\in \mc H^G(M;\mb Z_2)$, there exists an mass minimizer $\Sigma$ in $\alpha$, i.e. $\mf M(\llbracket\Sigma\rrbracket)=\inf\{\mf M(T): T\in\alpha\}$, so that $\Sigma$ is an essential minimal realization of $\alpha$. 
\end{lemma}

\begin{proof}
    Consider a minimizing sequence $T_i \in \alpha$ with $\lim_{i \rightarrow \infty} \mf{M}(T_{i}) = \inf\{\mf M(T): T\in\alpha\}$. By the compactness theorem, $T_{i}$ converges (up to a subsequence) to $\Sigma \in \mc {LB}^G(M;\mb Z_2)$. To get $\Sigma \in \alpha$, note that $T_i = T_{0} + \partial \Omega_{i}$ holds for each $i$, where $T_0 \in \mc {LB}^G(M;\mb Z_2)$ represents $\alpha$ and $\Omega_i \in \mathcal{C}^{G}(M)$ for all $i$. 
    Since $\sup_{i} \mathbf{M}(\Omega_i) \leq \mf M (M)< \infty$, the compactness theorem implies that $\Omega_i$ converges (up to a subsequence) to $\Omega \in \mathcal{C}^{G}(M)$, and $\partial \Omega_i\to\bd\Omega$. 
    Since $\partial \Omega_i = T_i - T_0$ also converges to $\Sigma - T_0$, we conclude that $\Sigma = T_0 + \partial \Omega$. This proves that $\Sigma \in \alpha$.    

    By the lower semi-continuity of the mass norm and $\Sigma \in \alpha$, 
    $\Sigma$ is a mass minimizer in $\alpha$. By applying \cite{wang2026homology}*{Theorem 2.11} locally with $h = 0$, $\Sigma$ is an embedded minimal $G$-hypersurface. As $\Sigma$ is mass-minimizing in $\alpha \neq [0]^{G}$, it is an essential minimal realization of $\alpha$.       
\end{proof}

\begin{remark}\label{Rem: minimizing in mfd with boundary}
	If $M$ is a compact Riemannian $G$-manifold with smooth boundary, then one can similarly define $[T]^G\in\mc H^G(M;\mb Z_2)$ (resp. $[T]^G\in \wti{\mc H}^G(M;\mb Z_2)$) for $T\in\mc {LB}^G(M;\mb Z_2)$ (resp. $T\in\mc Z_n^G(M;\mb Z_2)$) as the $G$-homology class of mod $2$ $G$-cycles. 
	Note that $S\in [T]^G$ forces $\bd S=0$, and $[T]^G$ is not the relative $G$-homology class.  
	If we further assume that $\bd M$ is minimal, then one can also construct area minimizing realization $\Sigma$ of $[T]^G\in \mc {LB}^G(M;\mb Z_2)$ as above so that every $G$-component of $\Sigma$ is contained either in $\bd M$ or in $\interior(M)$ by the maximum principle. 
\end{remark}

Note that $\mc H^G(M;\mb Z_2)$ has a group structure given by $[T]^G+[S]^G:=[T+S]^G$ with identity $[0]^G$ and inverse $([T]^G)^{-1}=[T]^G$. 
Hence, the following lemma indicates that $\mc H^G(M;\mb Z_2)\cong \mb Z_2^r$ for some finite $r\in\mb N$. 
\begin{lemma}\label{Lem: finite G-homology class}
	$\mc H^G(M;\mb Z_2)$ is finite. 
\end{lemma}
\begin{proof}
	By the equivariant triangulation \cite{verona1979triangulation, illman1983equivariant}, we know the orbit space $\pi(M)$ has a triangulation so that the orbit type is constant in each open simplex. 
	Hence, we have finite (simplicial) homologies $H_k(\pi(M),\mb Z_2)$, $H_k(\pi(M\setminus M^{prin});\mb Z_2)$, and $H_k(\pi(M), \pi(M\setminus M^{prin});\mb Z_2)$. 
 	Note that every $\alpha\in\mc H^G(M;\mb Z_2)$ with a realization $\Sigma$ can be uniquely associated to a relative homology class $[\pi(\Sigma)]\in H_l(\pi(M), \pi(M\setminus M^{prin});\mb Z_2)$ in the orbit space, where $l + 1$ is the dimension of $M/G$ (see \ref{Eq: cohomogeneity}). This shows the finiteness of $\mc H^G(M;\mb Z_2)$. 
\end{proof}

\begin{definition}\label{Def: 3-type G-homology}
	Given any $\alpha\in \mc H^G(M;\mb Z_2)$, we say 
	\begin{itemize}
		\item $\alpha$ is of type (I) if $T \in \mc B^G(M;\mb Z_2)$ for all $T\in\alpha$;
		\item $\alpha$ is of type (II) if $T \in  (\mc {LB}^G(M;\mb Z_2)\cap \mc B(M;\mb Z_2))\setminus \mc B^G(M;\mb Z_2)$ for all $T\in\alpha$;
		\item $\alpha$ is of type (III) if $T \in  \mc {LB}^G(M;\mb Z_2)\setminus \mc B(M;\mb Z_2)$ for all $T\in\alpha$. 
	\end{itemize}
\end{definition}

One easily verifies that every $\alpha\in \mc H^G(M;\mb Z_2)$ must be in one of the three types.
In \cite{wang2026homology}, it has been shown that the $G_\pm$-signed-symmetric constructions are required for $\alpha$ of type (II), and $\alpha$ of type (III) can be transformed into type (II) by considering the associated double cover with lifted group actions. 
To be exact, we have the following result. 
\begin{proposition}[\cite{wang2026homology}*{Proposition 2.4, 2.5}]\label{Prop: lift group action}
	Given $\alpha\in\mc H^G(M;\mb Z_2)$, 
    \begin{itemize}
        \item[(1)] if $\alpha$ is of type (II), then there is an index $2$ Lie subgroup $G_+ < G$ with $G_-:=G\setminus G_+$ so that for any $T\in\alpha$ and $p\in M$, we have $T\in \mc B^{G_\pm}(M;\mb Z_2)$, 
    		\begin{align}\label{Eq: free Gpm}
        			G\cdot p=G_+\cdot p \sqcup G_-\cdot p \qquad{\rm and}\qquad \mbox{$G_-$ permutes $\{G_+\cdot p,G_-\cdot p\}$},
    		\end{align}
        \item[(2)] if $\alpha$ is of type (III), then there exists a double cover $\tau:\wti M\to M$ and a trivial Lie group $2$-cover $p:\wti G\cong G\times\mb Z_2 \to G$ with $p^{-1}(e)\cong\mb Z_2$ so that 
        \begin{itemize}
            \item[(a)] $\wti G$ acts by isometries on $\wti M$ with $\tau\circ \tilde{g} = p(\tilde{g})\circ\tau: \wti M\to M$ for all $\tilde{g}\in\wti G$,
            \item[(b)] $p^{-1}(e)=\{\tilde{e},\tilde{g}_-^0\}$ is the deck transformation group of $\tau$, where $\tilde{e}\in \wti G$ is the identity, 
            \item[(c)] there is an index $2$ Lie subgroup $\wti G_+<\wti G$ with $\wti G_-:=\tilde{g}_-^0\cdot \wti G_+= \wti G\setminus \wti G_+$ so that $p|_{\wti G_+}:\wti G_+\to G$ is an isomorphism, and \eqref{Eq: free Gpm} holds in $\wti M$ with $\wti G_\pm$ in place of $G_\pm$; 
            \item[(d)] for every $T\in\alpha$, there exists $\wti T\in\mc B^{\wti G_\pm}(\wti M;\mb Z_2)$ so that $\tau_\#(|\wti T|) = 2|T|$.
        \end{itemize}
    \end{itemize}
\end{proposition}

\medskip
For any embedded closed hypersurface $\Sigma$ in the interior of a connected subset $U\subset M$, recall that $\Sigma$ is said to be {\em separating} if $U\setminus \Sigma$ has two connected components. 
Then every connected closed hypersurface $\Sigma\subset M$ is either $2$-sided separating, $2$-sided non-separating, or $1$-sided (also non-separating). 
Similarly, we use the following definitions. 

\begin{definition}\label{Def: equivariant separating}
	Let $\Sigma$ be an embedded closed $G$-connected $G$-hypersurface, and $U\subset M$ be a $G$-connected subset so that $\Sigma\subset\interior(U)$. 
	Then $\Sigma$ is said to be {\em $G$-separating in $U$} if $U\setminus \Sigma$ has two $G$-components. 
	In particular, there are exact four cases of $\Sigma$:
	\begin{itemize}
		\item[(1)] $\Sigma$ is $2$-sided $G$-separating in $U$;
		\item[(2)] $\Sigma$ is $2$-sided non-$G$-separating in $U$ with a $G$-invariant unit normal $\nu$;
		\item[(3)] $\Sigma$ is $2$-sided non-$G$-separating in $U$ and does not admit a $G$-invariant unit normal;
		\item[(4)] $\Sigma$ is $1$-sided. 
	\end{itemize}
    In Case (1)(2) (resp. (3)(4)), $\Sigma$ admits a (resp. no) $G$-invariant unit normal, and is said to be {\em $(G,2)$-sided} (resp. {\em $(G,1)$-sided}). 
\end{definition}

\begin{remark}\label{Rem: locally G-boundary-type for G,2-sided}
    For an embedded closed $G$-hypersurface $\Sigma\subset M$, if $\Sigma$ is $(G,2)$-sided, then $\Sigma$ is $G$-separating in a $G$-invariant small $r$-neighborhood $B_r(\Sigma)$, and thus $\llbracket\Sigma\rrbracket\in\mc {LB}^G(M;\mb Z_2)$. 
    If $\Sigma$ is $(G,1)$-sided, then the metric completion $\wti M$ of $M\setminus \Sigma$ has a new boundary $G$-component $\wti\Sigma$ that is a $G$-connected double cover of $\Sigma$ (diffeomorphic to $\bd B_r(\Sigma)$) so that $\wti \Sigma$ is $(G,2)$-sided and $\llbracket\wti\Sigma\rrbracket\in\mc {LB}^G(\wti M;\mb Z_2)$. 
    We refer to \cite{wang2024Ricci}*{(4.3)-(4.6)} for the specific constructions of $\wti M$ and $\wti\Sigma$. 
\end{remark}

\begin{remark}
	In Definition \ref{Def: equivariant separating}, note that the metric completion of $U\setminus \Sigma$ is $G$-equivariantly diffeomorphic to $U\setminus B_r(\Sigma)$ for some $r>0$ small enough. Hence, the above cases are analogous to the non-equivariant setting in the following sense:
	\begin{itemize}
		\item For $\Sigma$ in Case (1), $U\setminus \Sigma$ has two $G$-components, each of which has a new boundary $G$-component that is isometric to $\Sigma$. 
			Thus, Case (1) corresponds to the $2$-sided separating case in the non-equivariant setting of Song \cite{song2023infinite}. 
			Additionally, $[\Sigma]^G$ must be of type (I). 
		\item In Case (2), the metric completion of $U\setminus\Sigma$ is still $G$-connected, but has two new boundary $G$-components (diffeomorphic to $\bd B_r(\Sigma)$), each of which is isometric to $\Sigma$. 
			Hence, Case (2) corresponds to the $2$-sided non-separating case in the non-equivariant setting. 
			In this case, $[\Sigma]^G$ can be of type (II) (e.g. $\Sigma=S^2\times \{\pm p\}$ is two $2$-spheres in $S^2\times S^1$ with $\mb Z_2$ acting on $S^1$ by the antipodal map) or be of type (III) (e.g. $\Sigma$ is a $2$-sphere in $S^2\times S^1$ with $\mb Z_2$ acting on $S^2$ by the antipodal map). 
		\item In Case (3) and (4), the metric completion of $U\setminus\Sigma$ is $G$-connected with a single new boundary $G$-component that is isometric to a double cover of $\Sigma$. Additionally, since this new boundary $G$-component is equivariantly diffeomorphic to $\bd B_r(\Sigma)$, it is a trivial $2$-cover of $\Sigma$ in Case (3) and non-trivial in Case (4) respectively. 
			Therefore, both Case (3) and (4) correspond to the $1$-sided case in the non-equivariant setting. 
            Suppose $\llbracket\Sigma\rrbracket\in\mc {LB}^G(M;\mb Z_2)$.
			Then, in Case(4), $[\Sigma]^G$ must be of type (III). 
			However, in Case (3), $[\Sigma]^G$ can be of type (II) (e.g. $\Sigma$ is a $2$-sphere in $S^3$ with $\mb Z_2$ acting on $S^3$ by the antipodal map) or be of type (III) (e.g. $\Sigma=S^2\times \{e^{i\cdot 0}\}$ is a $2$-sphere in $S^2\times S^1$ with $\mb Z_2$ acting by $(p, e^{i\theta})\mapsto (-p, e^{-i\theta})$).  
	\end{itemize}
\end{remark}

\begin{lemma}\label{Lem: intersect G1-sided}
    For any two closed embedded $G$-hypersurfaces $\Sigma$ and $\Gamma$ in $ M$,
    if $\Sigma$ is $(G,1)$-sided and $[\Sigma]^G=[\Gamma]^G$, then $\Sigma\cap\Gamma\neq\emptyset$. 
\end{lemma}
\begin{proof}
    Otherwise, there exists an open $G$-set $\Omega$ with smooth boundary $\bd\Omega=\Sigma\sqcup\Gamma$. This gives a $G$-invariant (outward) unit normal along $\Sigma$ as a contradiction. 
\end{proof}

For closed minimal $G$-hypersurfaces in Case (1) and (2), we can use the implicit function theorem to generalize \cite{bray2010rigidity}* {Proposition 3.2} and \cite{song2023infinite}*{Lemma 10}. For $\Sigma$ in Case (3) and (4), one can also apply it in the metric completion of $B_r(\Sigma)\setminus\Sigma$ so that $\Sigma$ is lifted to its double cover with a $G$-invariant unit normal. 

\begin{lemma}\label{Lem: local foliation}
	Let $\Gamma$ be an embedded closed $(G,2)$-sided minimal $G$-hypersurface in $(M^{n+1}, g_{_M})$ with a $G$-invariant unit normal $\nu$. 
	Then there exist a positive number $\delta_1>0$ and a $G$-invariant smooth function $w:\Gamma\times (-\delta_1,\delta_1)\to \mb R$ (i.e. $w(g\cdot x,\cdot)=w(x,\cdot)$ for all $g\in G, x\in\Gamma$) so that 
	\begin{itemize}
		\item[(i)] for each $x\in \Gamma$, $w(x,0)=0$, and $\phi:=\frac{\bd}{\bd t}|_{t=0} w(x,t)$ is the first eigenfunction of the Jacobi operator $L_\Gamma$ of $\Gamma$; 
		\item[(ii)] for each $t\in (-\delta_1,\delta_1 ) $, we have $\int_\Gamma (w(\cdot,t)-t\phi)\phi = 0$;
		\item[(iii)] for each $t\in (-\delta_1,\delta_1)$, the mean curvature of the $G$-hypersurface $\Gamma_t:=\exp_{\Gamma}^\perp(w(\cdot,t)\nu)$ is either positive or negative or identically zero. 
	\end{itemize}
\end{lemma}

\begin{proof}
	Let $\phi>0$ be the first eigenfunction of the Jacobi operator $L_\Gamma$. 
	For any $g\in G$, since $dg(\phi\cdot\nu)$ is also the first eigenvector field and $\nu$ is $G$-invariant, we know that $\phi\circ g>0$ is also the first eigenfunction.
	Hence, $\phi\circ g= \phi$, i.e. $\phi$ is $G$-invariant, because the first eigenspace has dimension one and $\phi>0$. 
	If $\Gamma$ is unstable or strictly stable, then we simply take $w(x,t):=t\cdot \phi(x)$, which directly shows (i) and (ii). 
	Additionally, (iii) follows from the fact that $\frac{\bd}{\bd t}\vert_{t=0} \langle H_t, \nu \rangle = L_\Gamma\phi $ has a sign, where $H_t$ is the mean curvature vector field of $\Gamma_t$. 
	
	Next, consider the case that $\Gamma$ is degenerate stable, i.e. $L_\Gamma\phi=0$. 
	Let $X^G_1:=\{f\in C^\infty_G(\Gamma): \int_\Gamma f\phi =0 \}$ be the space of $G$-invariant functions that are orthogonal to $\phi$, and 
	\[\Phi: X^G_1\times [-1,1] \to X_1^G, \quad \Phi(u,t):= \frac{1}{\phi}\cdot \left( H(E(u+t\phi)) - \frac{\int_\Gamma H(E(u+t\phi))}{\Area(\Gamma)} \right), \] 
	where $E(u+t\phi):= \exp_\Gamma^\perp((u+t\phi)\nu)$, and $H$ is the mean curvature with respect to $\nu$. 
	Then,
	\[D_1\Phi_{(0,0)}(u)= \frac{1}{\phi}\cdot \left( L_\Gamma u - \frac{\int_\Gamma L_\Gamma u}{\Area(\Gamma)} \right).\]
	If $D_1\Phi_{(0,0)}(u)=0$, then $L_\Gamma u$ is a constant $c$, and $\int_\Gamma c\phi = \int_\Gamma \phi L_\Gamma u = \int_\Gamma u L_\Gamma\phi =0$, which implies $c=0$ and $u\in \{0\}= X_1^G\cap \{s\phi:s\in\mb R\}$. 
	Hence, the implicit function theorem indicates that for each $t\in (-\delta_1,\delta_1)$ with $\delta_1>0$ small enough, there exists $u_t\in X^G_1$ so that $\Phi(u_t,t)=0$. 
	Then $w(x,t):=u_t(x)+t\phi(x)$ satisfies (ii) and (iii). 
	Additionally, since $L_\Gamma\phi=0$, we have $D_2\Phi_{(0,0)}(1)=\frac{\bd}{\bd t}\vert_{t=0}\Phi(0,t) = 0$. 
	Combined with $0=\frac{\bd}{\bd t}\vert_{t=0}\Phi(u_t,t)=D_1\Phi_{(0,0)}(\frac{\bd u_t}{\bd t})+ D_2\Phi_{(0,0)}(1)$, we conclude that $D_1\Phi_{(0,0)}(\frac{\bd u_t}{\bd t})=0$, which gives (i).   
\end{proof}

Using the above lemma, we can make the following definitions. 
\begin{definition}\label{Def: 3-types neighborhood}
	Let $\hat M\subset M$ be a compact $G$-invariant domain with smooth boundary, and $\Sigma\subset \hat M$ be an embedded closed minimal $G$-hypersurface with a $G$-invariant unit normal $\nu$. 
	Let $\mu>0$, $\mc N$ be a $G$-neighborhood of $\Sigma$ in $M$, and $F$ be a $G$-equivariant diffeomorphism 
	\[ F: \Sigma\times (-\mu,\mu)\to \mc N, \quad F(g\cdot x, t)= g\cdot F(x, t),\]
	so that $F(x,0)=x$ for any $x\in\Sigma$. 
	Then we say that 
	\begin{itemize}
		\item[(i)] $\Sigma$ has a {\em contracting $G$-neighborhood} if there exist such $(\mu,\mc N, F)$ so that $F(\Sigma\times\{t\})$ has its mean curvature vector pointing towards $\Sigma$ for all $t\in (-\mu,\mu)\setminus\{0\}$;
		\item[(ii)] $\Sigma$ has a {\em expanding $G$-neighborhood} if there exist such $(\mu,\mc N, F)$ so that $F(\Sigma\times\{t\})$ has its mean curvature vector pointing away from $\Sigma$ for all $t\in (-\mu,\mu)\setminus\{0\}$;
		\item[(iii)] $\Sigma$ has a {\em mixed $G$-neighborhood} if there exist such $(\mu,\mc N, F)$ so that $F(\Sigma\times\{t\})$ has its mean curvature vector pointing towards (resp. away from) $\Sigma$ for all $t\in (-\mu,0)$ (resp. $t\in (0, \mu)$). 
	\end{itemize}
	If $\Sigma\subset\bd \hat M$, then we can similarly define the {\em contracting/expanding $G$-neighborhood in $\hat M$} as (i) and (ii) by considering $t\in (0,\mu)$ with $F(\Sigma\times \{t\})\subset \hat M$. 
	Moreover, if $\Sigma$ does not admit a $G$-invariant unit normal, then by considering the metric completion $\wti M$ of $\hat{M}\setminus\Sigma$, we have a new boundary $G$-component $\wti \Sigma$ that is a $2$-cover of $\Sigma$ with a $G$-invariant unit normal, and the {\em contracting/expanding $G$-neighborhood of $\Sigma$} can be defined as before using $\wti\Sigma\subset \bd \wti M$. 
\end{definition}

\section{Confined equivariant min-max minimal hypersurfaces}\label{Sec: confined min-max}

In this section, let $(N,g_{_N})$ be a compact Riemannian manifold with smooth boundary $\bd N$, and $G$ be a compact Lie group acting isometrically on $N$ satisfying \eqref{Eq: dimension assumption}. 
By \cite{wang2023free}*{Lemma A.1}, $N$ can be regarded as a compact $G$-domain in a closed Riemannian $G$-manifold $M$. 
Suppose $\bd N$ is a minimal $G$-hypersurface with a contracting $G$-neighborhood in $N$ (see Definition \ref{Def: 3-types neighborhood}). 
Namely, there exists a $G$-neighborhood $\mc N\subset N$ of $\bd N$ and a $G$-equivariant diffeomorphism 
\begin{align}
	F: \bd N \times [0, \hat{t} ] \to \mc N \quad {\rm with}\quad F(g\cdot x, t)=g\cdot F(x,t)
\end{align}
so that $F(\bd N\times \{0\})=\bd N$, and $F(\bd N\times \{t\})$ has its mean curvature pointing towards $\bd N$ for all $t\in (0,\hat t)$. 
Then we can define a non-compact manifold with $G$-equivariant cylindrical ends:
\begin{align}
	Cyl(N) := N \cup_\varphi (\bd N \times [0,\infty)),
\end{align}
where $\varphi: \bd N\times \{0\}\to \bd N$ is the canonical identity map, and $G$ acts on $\bd N \times [0,\infty)$ by $g\cdot (p,t)=(g\cdot p, t)$ for any $g\in G$, $p\in \bd N$ and $t\in [0,\infty)$. 
Extend the metric $h$ on $Cyl(N)$ by 
\begin{align}
	h := g_{_N}\llcorner_{\bd N} \oplus dt^2 \quad \mbox{on $\bd N \times [0,\infty)$}, \qquad {\rm and}\qquad h=g_{_N} \quad \mbox{on $N$},
\end{align}
which is $G$-invariant and Lipschitz continuous.

\subsection{Non-compact equivariant manifold with cylindrical ends}\label{section: construction of approximations} 
~

Similarly as in \cite{song2023infinite}, we define for any small $\epsilon > 0$ a compact Riemannian manifold with boundary $(N_{\epsilon}, h_{\epsilon})$ that is diffeomorphic to $N$, and converges geometrically to $(Cyl(N), h)$ in the $C^{0}$-topology as $\epsilon \rightarrow 0$. 

For $0 < \epsilon < \hat{t}$, let us set 
\[
\widetilde{N}_{\epsilon}:= N \setminus F(\partial N, [0, \epsilon)) \quad \text{with} \quad \partial \widetilde{N}_{\epsilon} = F(\partial N \times \{\epsilon\}). 
\]
By construction, $\partial \tilde{N}_{\epsilon}$ is a $G$-hypersurface with a natural $G$-invariant inward unit normal $\nu$. For a small positive number $\delta_{\epsilon} > 0$, the following map  
\begin{align*}
\tilde{\gamma}_{\epsilon}: \partial \tilde{N}_{\epsilon} \times [-\delta_\epsilon, 0] \rightarrow \mc N\\
\tilde{\gamma}_{\epsilon}(x, t) := \exp_{g_{N}}(x, t\nu) 
\end{align*}
is a well-defined $G$-equivariant diffeomorphism 
and gives the Fermi coordinates on one side of $\partial \tilde{N}_{\epsilon}$. By choosing $\delta_{\epsilon}>0$ small enough, we can ensure that 
\begin{itemize}
    \item $\lim_{\epsilon \rightarrow 0} \delta_{\epsilon} = 0$, 
    \item $\tilde{\gamma}_{\epsilon}(\partial \widetilde{N}_{\epsilon} \times [-\delta_{\epsilon}, 0]) \subset F(\partial N \times [0, \epsilon])$, 
    \item for $t \in [-\delta_\epsilon, 0]$, the hypersurface $\tilde{\gamma}_{\epsilon}(\partial \widetilde{N}_{\epsilon} \times \{t\})$ has mean curvature pointing towards $\partial N$.    
\end{itemize}

Given the notations above, set $N_0 = N$ and define for $0 < \epsilon < \hat{t}$  
\[
N_{\epsilon} := \widetilde{N}_{\epsilon} \cup \tilde{\gamma}_{\epsilon}(\partial \widetilde{N}_{\epsilon} \times [-\delta_{\epsilon}, 0]).  
\]
By construction, $\partial N_{\epsilon}$ is a $G$-hypersurface with a natural $G$-invariant inward unit normal $\nu$. For a small positive number $\hat{d} > 0$ independent of $\epsilon$, we have the Fermi coordinates on a $\hat{d}$-neighborhood of $\partial N_{\epsilon}$ given by the $G$-equivariant diffeomorphism   
\begin{align*}
\gamma_{\epsilon}: \partial N_{\epsilon} \times [-\hat{d}, \hat{d}] \rightarrow M\\
\gamma_{\epsilon}(x, t) = \exp_{g_{N}}(x, t\nu).  
\end{align*}
Note that for all $s \in [0, \delta_{\epsilon}]$, $\gamma_{\epsilon}(\partial N_{\epsilon} \times \{s\}) = \widetilde{\gamma}_{\epsilon}(\partial \tilde{N}_{\epsilon} \times \{-\delta_{\epsilon} + s\})$. Moreover, in terms of the Fermi coordinates $\gamma_{\epsilon}$, the original metric $g_{N}$ can be expressed as $g_{t} \oplus dt^2$.   

Next, for each small $\epsilon$, pick $z_{\epsilon} \in (0, \delta_{\epsilon})$ and a smooth function $\vartheta_{\epsilon}: [0, \delta_\epsilon] \rightarrow \mathbb{R}$ satisfying that  
\begin{itemize}
    \item $\vartheta_{\epsilon} \geq 1$ and $\frac{d}{dt} \vartheta_{\epsilon} \geq 0$, 
    \item $\vartheta_{\epsilon} \equiv 1$ in a neighborhood of $\delta_{\epsilon}$, 
    \item $\vartheta_{\epsilon}$ is constant on $[0, z_{\epsilon}]$,
    \item $\lim_{\epsilon \rightarrow 0} \int_{[0, \delta_{\epsilon}]} \vartheta_{\epsilon} = \infty$,
    \item $\lim_{\epsilon \rightarrow 0} \int_{[z_\epsilon, \delta_{\epsilon}]} \vartheta_{\epsilon} = 0$. 
\end{itemize}
There is a naturally induced $G$-invariant function  $\vartheta_{\epsilon}$ on $N_{\epsilon}$ defined by 
\[
\vartheta_{\epsilon}(\gamma_{\epsilon}(x, t)) = \vartheta_{\epsilon}(t) \quad \text{for all } (x, t) \in \partial N_{\epsilon} \times [0, \delta_{\epsilon}],  
\]
and extended continuously by $1$. Now we endow $N_{\epsilon}$ with the $G$-invariant smooth metric $h_{\epsilon}$: 
\[
h_{\epsilon}(x) := 
\begin{cases}
g_t(q) \oplus (\vartheta(q) dt)^2 \quad \text{for } q \in \gamma_{\epsilon} (\partial N_{\epsilon} \times [0, \delta_{\epsilon}]), \\
g_{N}(q) \quad \text{for } q \in N_{\epsilon} \setminus \gamma_{\epsilon} (\partial N_{\epsilon} \times [0, \delta_{\epsilon}]).
\end{cases}
\]
This defines the desired compact manifold with boundary $(N_{\epsilon}, h_{\epsilon})$. For our later purposes, we remark that $F(\partial N \times \{t\})$ for $t \in [\epsilon, \hat{t}]$ together with $\gamma_{\epsilon}(\partial N_{\epsilon} \times \{t\})$ for $t \in [0, \delta_{\epsilon}]$ forms a continuous strictly mean-concave foliation of $\partial N_{\epsilon}$, i.e. each leaf is smooth with non-zero mean curvature vector pointing towards $\partial N_{\epsilon}$ w.r.t. the metric $h_{\epsilon}$.

By \cite{song2023infinite}*{Lemma 6}, the following approximation result holds: $(N_{\epsilon}, h_{\epsilon}, q)$ converges geometrically to $(Cyl(N), h, q)$ in the $C^0$-topology as $\epsilon \rightarrow 0$. Moreover, the geometric convergence is smooth outside of $\partial N \subset Cyl(N)$. 
\subsection{Equivariant min-max theory in manifolds with cylindrical ends}\label{section: equivariant min-max theory in manifolds with cylindrical ends} 
~

Let $(N, g_{_N})$ be a compact Riemannian $G$-manifold with boundary $\bd N$, which is isometrically embedded as a compact $G$-domain in a closed $G$-manifold $M$. 
Define $\mc Z_n^G(N,\bd N;\mb Z_2)$ to be the space of $T\in \mf I_n^G(N;\mb Z_2)$ with $\spt(\bd T)\subset  \bd N$. 
Two elements $T,S \in \mc Z_n^G(N,\bd N;\mb Z_2)$ are said to be equivalent if $T-S\in \mf I_n^G(\bd N;\mb Z_2)$. 
Then the space of {\em relative $G$-cycles} $\mc Z^G_{n,rel}(N,\bd N;\mb Z_2)$ is defined by the space of such equivalence classes $[T]_{rel}$. 
Note that $\mc Z^G_{n,rel}(N,\bd N;\mb Z_2)$ coincides with $\mc Z^G_n(N;\mb Z_2)$ when $\bd N=\emptyset$, and the mass norm $\mf M$ and the flat metric $\mc F$ naturally defined on $\mc Z^G_{n,rel}(N,\bd N;\mb Z_2)$ as in \cite{almgren1962homotopy}*{Definition 1.20}. 
Similarly, we also define the following subspaces 
\begin{itemize}
	\item $\mc B^G_{rel}(N,\bd N;\mb Z_2):=\{[\bd\Omega]_{rel}\in \mc Z^G_{n,rel}(N,\bd N;\mb Z_2): \Omega \in \mc C^G(N) \}$;
	\item $\mc B^{G_\pm}_{rel}(N,\bd N;\mb Z_2):=\{[\bd\Omega]_{rel}\in \mc Z^G_{n,rel}(N,\bd N;\mb Z_2): \Omega \in \mc C^{G_\pm}(N) \}$;
	\item $\mc {LB}^G_{rel}(N,\bd N;\mb Z_2)$ is the space of $\tau\in \mc Z^G_{n,rel}(N,\bd N;\mb Z_2)$ so that for any $p\in N$, there exist $r\in (\inj(G\cdot p)/2, \inj(G\cdot p))$ and $\Omega\in C^G(B_r(G\cdot p))$ with $[\bd\Omega]_{rel}=\tau$ in $B_r(G\cdot p)$. 
\end{itemize}
For any $\tau \in \mc {LB}^G_{rel}(N,\bd N;\mb Z_2)$, define the {\em relative $G$-homology class} of $\tau$ by 
\begin{align*}
	[\tau]^G &:= \tau + \mc B^G_{rel}(N,\bd N;\mb Z_2) 
	\\
	&= \{\sigma \in \mc {LB}^G_{rel}(N,\bd N;\mb Z_2): \sigma = \tau + [\bd\Omega]_{rel} \mbox{ for some } \Omega\in \mc C^G(N) \},  
\end{align*}
which has three types (I)-(III) similar to Definition \ref{Def: 3-type G-homology}. 
In particular, using an almost verbatim argument, Lemma \ref{Lem: minimizer representative} and Proposition \ref{Prop: lift group action} are also valid for $[\tau]^G$. 

Note that a map $\Phi$ into $[\tau]^G$ is continuous in the $\mf M$ (resp. $\mc F$) topology if and only if $\Phi':= \Phi - \tau$ is an $\mf M$-continuous (resp. $\mc F$-continuous) map into $\mc B^G_{rel}(N,\bd N;\mb Z_2)$. 
Hence, by the equivariant version of Almgren's isomorphism \cite{wang2023free}*{(4.7)}, $[\tau]^G$ is weakly homotopically equivalent to $\mb {RP}^\infty$. 
Denote by $\bar{\lambda}$ the generator of $H^1([\tau]^G;\mb Z_2)\cong \mb Z_2$.

Let $X$ be a finite dimensional simplicial complex and $p\in \mb N$. 
An $\mc F$-continuous map $\Phi: X\to [\tau]^G$ is said to be a {\em $(G,p)$-sweepout} in $[\tau]^G$ if 
\[ \Phi^*(\bar{\lambda}^p)\neq 0 \in H^p(X;\mb Z_2). \]
Denote by $\mc P_p([\tau]^G)$ the set of all $(G,p)$-sweepouts in $[\tau]^G$ that have {\em no concentration of mass on orbits}, i.e. $\lim_{r\to 0}\sup\{\mf M(\Phi(x)\llcorner B_r(G\cdot p) ) : x\in {\rm dmn}(\Phi), p\in N\} = 0$. 
Then the {\em $(G,p)$-width in $[\tau]^G$} is 
\begin{align}
	\omega_p([\tau]^G,  g_{_N}) := \inf_{\Phi\in \mc P_p([\tau]^G)}\sup_{x\in {\rm dmn}(\Phi)} \mf M(\Phi(x))
\end{align}
Note that $|\omega_p([\tau]^G,  g_{_N}) - \omega_p([0]^G,  g_{_N})|\leq \mf M_{g_{_N}}(\tau)$. Hence, by \cite{wang2025density}*{Theorem 1.5},
\begin{align}\label{Eq: Weyl law}
    \omega_p([\tau]^G,  g_{_N})\sim p^{\frac{1}{l+1}} \qquad \mbox{as $p\to\infty$},
\end{align}
where $l+1$ is the cohomogeneity defined in \ref{Eq: cohomogeneity}. 

\begin{remark}
	The above equivariant constructions can also be made using integer rectifiable currents as in \cite{wang2023free}, which is indeed equivalent (c.f. \cite{guang2021min}*{Proposition 3.2} and \cite{wang2025density}*{Lemma 2.1}). 
\end{remark}

In non-compact $G$-manifolds, we use the following definitions of compactly supported $G$-homology classes and $(G,p)$-widths. 
\begin{definition}\label{Def: G-width in non-compact mfd}
	Let $(\wti N^{n+1},g_{_{\wti N}})$ be a complete non-compact manifold. 
	Define $\mc {LB}^G_c(\wti N;\mb Z_2)$ to be the space of local $G$-boundaries (c.f. \eqref{Eq: local boundary}) with compact support in $\wti N$. 
	Given $T\in \mc {LB}^G_c(\wti N;\mb Z_2)$, define the {\em compactly supported $G$-homology class} of $T$ by
	\[ [T]^G_c :=\{S\in \mc {LB}^G_c(\wti N;\mb Z_2): S=  T +\bd\Omega \mbox{ for some compactly supported } \Omega \in \mc C^G(\wti N)\} .\] 
	Let $K_1\subset K_2\subset \dots \subset K_i \subset \cdots$ be an exhaustion of $\wti N$ by compact $G$-invariant domains with smooth boundary so that $\spt(T)\subset\subset K_1$. 
	Then the {\em $(G,p)$-width} of $\wti N$ in $[T]^G_c$ is defined by 
	\[ \omega_p([T]^G_c, g_{_{\wti N}}) := \lim_{i\to\infty} \omega_p([\tau_i]^G, g_{_{\wti N}} )\in [0,\infty], \]
	where $\tau_i:= [T]_{rel}\in \mc {LB}^G_{rel}(K_i,\bd K_i;\mb Z_2)$, and $[\tau_i]^G := \tau_i + \mc B^G_{rel}(K_i,\bd K_i; \mb Z_2) $. 
\end{definition}

For $G$-sets with smooth boundary, one can adapt \cite{liokumovich2018weyl}*{Lemma 2.15 (1)} to equivariant settings as in \cite{wang2025density}*{Claim 1}, which implies $\omega_p([\tau_i]^G,g_{_N})$ is non-decreasing. 
Hence, $\omega_p([T]^G_c, g_{_{\wti N}})$ is well-defined and does not depend on the choices of $\{K_i\}$ or the representative $T\in [T]^G_c$.

\medskip
Using the above notations, we can now consider the equivariant min-max theory in the non-compact $G$-manifold $(Cyl(N), h)$ with cylindrical ends $\bd N\times [0,\infty)$ so that $\bd N$ is a minimal $G$-hypersurface with a contracting $G$-neighborhood in $N$. 
Let $\{\Sigma_i\}_{i=1}^m$ be the $G$-connected components of $\bd N$ so that $\Sigma_1$ has the largest area, i.e. 
\[ \Area(\Sigma_1) \geq \Area(\Sigma_{i}) \quad \mbox{for all $2\leq i \leq m$}.  \]
Given any smooth embedded closed $G$-hypersurface $\Gamma\subset \subset N$ with $\llbracket\Gamma\rrbracket \in \mc {LB}^G_c(Cyl(N) ; \mb Z_2)$, let $[\Gamma]^G_c$ be the compactly supported $G$-homology class of $\llbracket\Gamma\rrbracket $. 

We now show the following estimates for $\omega_p([\Gamma]^G_c, h)$. 

\begin{theorem}\label{Thm: width estimates in cylindrical mfd}
	There exists a constant $C=C(N,G,[\Gamma]^G_c, h)$ so that for all $p\in \mb N$, 
	\begin{itemize}
		\item[(i)] $\omega_{p+1}([\Gamma]^G_c) \geq \omega_p([\Gamma]^G_c)+ \Area(\Sigma_1)$;
		\item[(ii)] $p\Area(\Sigma_1 ) \leq   \omega_{p}([\Gamma]^G_c) \leq p\Area(\Sigma_1 ) +  C p^{\frac{1}{l+1}}$. 
	\end{itemize}
\end{theorem}

\begin{proof} 
    We closely follow the proof of \cite{song2023infinite}*{Theorem 9}, with some alterations. Thus, we only describe the necessary modifications. 

    To prove (i), recall that for $R>0$ large enough, the proof of \cite{song2023infinite}*{Theorem 9} indicates
    \[
    \omega_1^G(\Sigma_1 \times [0, R]) = \Area(\Sigma_1), 
    \]
    where $\omega_1^G(\Sigma_1 \times [0, R]):=\omega_1([0]^G)$ with $[0]^G\in \mc H^G(\Sigma_1 \times [0, R];\mb Z_2)$ is the ($G$-boundary-type) Almgren-Pitts $(G,1)$-width (c.f. \cite{wang2025density}*{Definition 3.4}) in $\Sigma_1 \times [0, R]$.   
    Given $\epsilon > 0$, choose $x_{0} \in Cyl(N)$ and $R_{p} > 0$ large enough so that for all $R \geq R_{p}$, $B_{R}:=\closure(B_R(G\cdot x_0))$ is a compact $G$-invariant subset with $\spt(\Gamma) \subset \subset B_{R}$ and    
    \[
    \omega_{p}([\tau_{B_{R}}]^G) \geq \omega_{p}([\Gamma]^G_c) - \epsilon, 
    \]
    where $\tau_{B_{R}} \in \mc {LB}^G_{rel}(B_{R},\bd B_{R};\mb Z_2)$ is induced by $\Gamma$, and 
    \begin{align*}
    [\tau_{B_{R}}]^G &:= \tau_{B_{R}} + \mc B^G_{rel}(B_{R},\bd B_{R}; \mb Z_2).
    \end{align*}
    
    Choose $R > R_p$ so that $B_{R}$ contains the disjoint union of $B_{R_p}$ and a subset $E_1$ isometric to $\Sigma_1 \times [0, R]$. By inspecting the proof of  \cite{liokumovich2018weyl}*{Theorem 3.1} (also see the proof of \cite{wang2025density}*{Theorem 4.6}), an analogous version of the Lusternik-Schnirelmann inequality holds in our setting and yields that 
    \begin{align*}
        \omega_{p+1}([\Gamma]^G_c) ~\geq~ \omega_{p}([\tau_{B_{R}}]^{G}) + \omega_{1}^G(E_1) -2\epsilon ~\geq~ \omega_{p}([\Gamma]^G_c) + \Area(\Sigma_1) -  3\epsilon. 
    \end{align*}
    Letting $\epsilon \rightarrow 0$ completes the proof of (i). 
    \medskip 
    
    We proceed to prove (ii) using the gluing technique of Liokumovich-Marques-Neves \cite{liokumovich2018weyl}. By assumptions, we have $Cyl(N) \setminus A = \partial N \times [0, \infty)$ where $A$ is a compact $G$-invariant domain with smooth boundary containing $\Gamma$.  Define $B_{L} := \partial N \times [0, L]$ for $L > 0$, whose boundaries intersect the boundaries of $A$ along the closed $G$-hypersurface $\partial N \times \{0\}$. 
    Denote by $\tau_A$ and $\tau_{A\cup B_L}$ the relative $G$-cycles induced by $\Gamma$ in $A$ and $A\cup B_L$ respectively. 
    Fix $p \in \mathbb{N}$, and
    we know from the Weyl law for equivariant volume spectrum (c.f. \eqref{Eq: Weyl law} and \cite{wang2025density}*{Theorem 1.5}) that there exist a uniform constant $C>0$ and a $(G, p)$-sweepout $\Phi_{1}: \mb{RP}^{p} \rightarrow [\tau_A]^G $ in $A$ with $\Phi_1 \in \mc P_p([\tau_A]^G)$ satisfying that  
    \[
    \max_{x \in \mb{RP}^{p}} \mathbf{M}(\Phi_{1}(x)) \leq C p^{\frac{1}{l + 1}}, 
    \]
    where $l+1$ is the cohomogeneity \eqref{Eq: cohomogeneity}. 
    Additionally, by the proof of \cite{song2023infinite}*{Theorem 9} there exist a $G$-equivariant Morse function (in the sense of \cite{wasserman1969equivariant}) $f: B_{L} \rightarrow \mathbb{R}$ and a map $\Phi_{2}: \mb{RP}^{p} \rightarrow \mathcal{B}_{rel}^{G}(B_L, \partial B_L; \mathbb{Z}_2)$ given by
    \[
    \Phi_{2}([a_{0}, a_{1}, \ldots, a_{p}]) = [\partial\{x\in B_L: a_{0} + a_{1} f(x) + \cdots + a_{p} f(x)^{p} < 0\}]_{rel} 
    \]
    so that $\Phi_2$ is a ($G$-boundary-type) $(G,p)$-sweepout satisfying that 
    \[
    \max_{x \in \mathbb{RP}^{p}} \mathbf{M}(\Phi_{2}(x)) \leq p \Area(\Sigma_1).  
    \]
    To glue the $p$-sweepouts $\Phi_1 $ and $\Phi_2$, we follow the proof in \cite{song2023infinite}*{Theorem 9, Page 875} with $\mc C^G,\mc B^{G}_{rel}$ in place of $\mf I_{n+1},\mc Z_{n,rel}$ respectively, use $\Phi_1(x)-\tau_A - \bd Z$ in the definition of $SX_0$, and add $\tau_A$ to the $G$-boundary-type map $\varkappa_L$. 
    Then we similarly obtain a $(G,p)$-sweepout $\Phi: \mb{RP}^{p} \rightarrow [\tau_{A \cup B_{L}}]^G \in \mc{P}_p([\tau_{A \cup B_{L}}]^G)$ such that for some constant $C=C(N,G,[\Gamma]^G_c, h) > 0$,  
    \begin{align*}
        \max_{x \in \mb{RP}^{p}} \mathbf{M}(\Phi(x)) &\leq \max_{x \in \mb{RP}^{p}} \mathbf{M}(\Phi_{1}(x)) +  
        \max_{x \in \mb{RP}^{p}} \mathbf{M}(\Phi_{2}(x)) + \Area(\bd N)\\ 
        &\leq C p^{\frac{1}{l + 1}} +  p\Area(\Sigma_1) + C\\
        &\leq C p^{\frac{1}{l + 1}} +  p\Area(\Sigma_1).  
    \end{align*}
It follows that 
\[
\omega_{p}([\Gamma]^G_c) = \lim_{L \rightarrow \infty} \omega_{p}([\tau_{A \cup B_{L}}]^G) \leq p\Area(\Sigma_1 ) +  C p^{\frac{1}{l+1}},    
\]
which completes the whole proof. 
\end{proof}

Using the approximations $(N_\epsilon, h_\epsilon)$, we show that $\omega_p([\Gamma]^G_c)$ can be realized by the area of some closed minimal $G$-hypersurfaces in the interior of $N$. 

\begin{theorem}\label{Thm: confined min-max hypersurface}
	Let $(Cyl(N), h)$ and $\Gamma\subset\subset N$ be given as before. 
	For any $p\in\mb N$, there exist disjoint smooth closed $G$-connected embedded minimal $G$-hypersurfaces $\{\Gamma_i\}_{i=1}^{l_p}$ contained in $N\setminus \bd N$ and $\{m_i\}_{i=1}^{l_p}\subset \mb N$ so that 
	\[ \omega_p([\Gamma]^G_c,h) = \sum_{i=1}^{l_p} m_i \Area(\Gamma_i). \] 
	Moreover, let $\tau:=[\Gamma]_{rel}\in \mc {LB}^G_{n,rel}(N,\bd N;\mb Z_2)$ and $[\tau]^G:=\tau+\mc B^G_{rel}(N,\bd N;\mb Z_2)$. We also have
	\begin{itemize}
		\item[(i)] if $\Gamma$ is $G$-separating in $N$, then every $(G, 1)$-sided $\Gamma_i$ has even multiplicity $m_i$;
		\item[(ii)] if $\Gamma$ is not $G$-separating in $N$, then there is a subset $\mc I\subset\{1,\dots, l_p\}$ so that $\llbracket \cup_{i\in \mc I}\Gamma_i \rrbracket_{rel} \in [\tau]^G$. In particular, $\llbracket \Gamma_1 \rrbracket_{rel} \in [\tau]^G$ has odd multiplicity $m_1$ when $l_p = 1$.  
	\end{itemize}
\end{theorem}

\begin{proof}
Recall that in Section \ref{section: construction of approximations} we have constructed compact smooth approximations $(N_{\epsilon}, h_{\epsilon})$ of $(Cyl(N), h)$. Fix $p \in \mathbb{N}$. Applying the $G$-equivariant free boundary min-max theorem (cf. \cite{wang2023free}) gives a $G$-invariant varifold $V_{\epsilon}$ of total mass close to $\omega_{p}([\Gamma]_c^{G})$.  
Moreover, $\spt (\|V_{\epsilon}\|) = \Gamma_{\epsilon}$ is a compact, almost properly embedded free boundary minimal $G$-hypersurface. Note that the $G$-boundary-type assumption is only used locally for the regularity in \cite{wang2023free}, which can be easily adapted for local $G$-boundaries.
As $F(\partial N \times \{t\})$ for $t \in [\epsilon, \hat{t}]$ together with $\gamma_{\epsilon}(\partial N_{\epsilon} \times \{t\})$ for $t \in [0, \delta_{\epsilon}]$ forms a strictly mean-concave foliation of $\partial N_{\epsilon}$, any component of $\Gamma_{\epsilon}$ that intersects $\gamma_{\epsilon}(\partial N_{\epsilon} \times [0,\delta_\epsilon])$ must also intersect $F(\partial N \times \{\hat{t}\})$. 
By the monotonicity formula, there exist $q \in N \setminus \partial N$ and $R > 0$ such that $\Gamma_{\epsilon} \subset B_{R}(G \cdot q)$ w.r.t. the metric $h_{\epsilon}$ for all $\epsilon > 0$ small enough. In particular, $\Gamma_\epsilon$ is a closed embedded minimal $G$-hypersurface in $\interior(N_\epsilon)$.

As $\epsilon \rightarrow 0$, we have $\omega_{p}([\tau_{N_\epsilon}]^G,h_\epsilon) \rightarrow \omega_{p}([\Gamma]^{G}_c, h)$, where $\tau_{N_\epsilon}:=[\Gamma{\llcorner N_\epsilon}]_{rel}\in \mc {LB}^G_{rel}(N_\epsilon,\bd N_\epsilon;\mb Z_2)$. Then up to a subsequence $\epsilon_k\to 0$, 
\[\lim_{k\to\infty} V_{\epsilon_k}= V_\infty \in \mc V^G_n(Cyl(N)) \quad {\rm with}\quad \|V_\infty\|(Cyl(N))=\omega_{p}([\Gamma]^{G}_c,h).\]
Since each $V_{\epsilon_k}$ is produced by min-max, we know that $V_{\epsilon_k}$ is automatically a min-max $(c, G)$-hypersurface as in Definition 3.10   
of \cite{wang2025density} for some $c=c(p) \in \mathbb{N}$ independent of $\epsilon$.   
Hence for any $c$-admissible family of $G$-annuli $\mathcal{A}$ (i.e. $\mathcal{A}=\{A_{s_i,t_i}(G\cdot x)\}_{i=1}^c$ with $0<s_i<t_i<s_{i+1}/2<t_{i+1}/2$), each $V_{\epsilon_k}$ has good $G$-replacement property in some $A \in \mathcal{A}$. Recall that as $\epsilon \rightarrow 0$, $(N_{\epsilon}, h_{\epsilon}, q)$ converges geometrically to $(Cyl(N), h, q)$ in the smooth topology outside of $\partial N \subset Cyl(N)$. Following the proof of \cite{wang2025density}*{Theorem 3.13} (also see \cite{wang2023G-index}*{Remark 5.5}), we know that $V_\infty \llcorner (Cyl(N)\setminus \partial N)$ has good $G$-replacement property in some $A \in \mathcal{A}$ for any $c$-admissible family of $G$-annuli $\mc A$ in $Cyl(N)\setminus \partial N$, i.e. $V_\infty \llcorner (Cyl(N)\setminus \partial N)$ is a min-max $(c, G)$-hypersurface. 
Then the regularity result \cite{wang2023G-index}*{Theorem 4.18} implies that the restriction of $\spt (\|V_\infty\|) = \Sigma$ to $Cyl(N)\setminus \partial N$ is a minimal $G$-hypersurface. The maximum principle says that if $\Sigma \cap (Cyl(N) \setminus N) \neq \emptyset$, $\Sigma$ would be a connected component of some slice $\partial N \times \{\delta\}$, which is a contradiction since $\Sigma$ now has to intersect $F(\bd N\times \{\hat{t}\})$. 
As a consequence, $\Sigma$ is contained in the compact set $(N, g_N)$.  

Next, we proceed as in \cite{song2023infinite} and show that $V_\infty$ is $g_N$-stationary. Denote by $\divergence^0$ the divergence computed in the metric $g_N$. It suffices to check that 
\[
\delta V_{\infty} = \int \divergence^0_S X(x) dV_{\infty}(x, S) = 0,  
\]
where $X$ is a vector field supported near and smooth up to $\partial N$. By choosing the same variational vector field $Y^{\epsilon}$ as in \cite{song2023infinite}*{P.19} and using the $h_{\epsilon}$-stationarity of $V_{\epsilon}$ (equivalent to the $G$-stationarity of $V_{\epsilon}$ w.r.t. the metric $h_\epsilon$ \cite{liu2021existence}*{Lemma 2.2}), we obtain that 
\[
V_{\infty} \llcorner \{(x, S): x \in \partial N, S \neq T_x \partial N\} = 0. 
\]
Now split $X$ as the orthogonal part $X_{\perp}$ and the parallel part $X_{||}$ w.r.t. $\frac{\partial}{\partial t}$. Based on the $h_{\epsilon}$-stationarity of $V_{\epsilon}$ and careful estimates using the fact above, we conclude that 
\[
\delta V_{\infty} = \int \divergence^0 X^{0}_{\perp} dV_{\infty} + \int \divergence^0 X^{0}_{||} dV_{\infty} = 0.  
\]
We refer to \cite{song2023infinite}*{Theorem 10, Step 3,4} for the details. 
From \cite{song2023infinite}*{Proposition 3} we further obtain $V_{\infty}$ as a $g_N$-stationary integral varifold. Then the maximum principle by White \cite{white2010maximum} implies that $\spt(\|V_\infty\|)=\Sigma=\cup_{i=1}^{l_p}\Gamma_i$ is a closed minimal $G$-hypersurface confined in the core $N \setminus \partial N$. In particular, it follows from the convergence $V_{\epsilon_k}\to V$ in $\interior(N)$ and the monotonicity formula that $\Gamma_{\epsilon_k}=\spt(||V_{\epsilon_k}||) $ is contained in a $G$-domain $U_{\epsilon_k}\subset \interior(N_{\epsilon_k})$ that is isometric to some $ U_k\subset \interior(N)$ for $k$ large enough.

\medskip
Suppose that $\Gamma$ is $G$-separating in $N$, and $\Gamma_i$ is a $(G,1)$-sided component of $\spt(\|V_\infty\|)$.  
We will follow the general scheme of \cite{zhou2015minmaxRicci}*{Section 6} (also see \cite{wang2024Ricci}*{Theorem 3.8}) to show that $\Gamma_i$ has even multiplicity $m_i$. Since $V_{\infty}$ is a min-max $(c, G)$-hypersurface supported in $\interior(N)$, we know that for any $c$-admissible family of $G$-annuli $\mathcal{A}$ in $N$, $V_{\infty}$ has good $G$-replacement property in some $A \in \mathcal{A}$. By \cite{wang2023G-index}*{Lemma 4.17}, for any $p \in N$, there exists $r(G \cdot p) > 0$ so that $V_\infty$ has good $G$-replacement property in $A_{s, t}(G \cdot p)$ for all $0 < s < t < r(G \cdot p)$. Hence we can take a point $p \in \Gamma_i$ and $r > 0$ small enough so that $V_{\infty}$ has good $G$-replacement property in $B_{2r}(G \cdot p)\subset \interior(N)$, $\bd B_t(G\cdot p)$ is strictly mean convex for $t\in (0,2r)$, and 
\[
    \spt(||V_\infty||) \cap B_{2r}(G \cdot p) =: \Sigma \cap B_{2r}(G \cdot p) = \spt(||\Gamma_i||) \cap B_{2r}(G \cdot p).
\]
Let $s \in (\frac{r}{2}, r)$ so that $\bd B_s(G\cdot p)$ is transversal to $\spt(\|V_\infty\|)$, and $V^{*}$ be a $G$-replacement of $V_{\infty}$ in $\Clos(B_s(G \cdot p))$. 
Then by \cite{wang2023G-index}*{Proposition 4.19}, $V^{*} = V_{\infty}$. Recall that the min-max $(c,G)$-hypersurface $V_\infty$ is given by the compactness theorem \cite{wang2025density}*{Theorem 3.13} (see also \cite{wang2023G-index}*{Theorem 5.3}) applied to $V_{\epsilon_k}$. 
Hence, it follows from the proof of the compactness theorem \cite{wang2025density}*{Theorem 3.13} that $V^*=\lim_{k\to \infty} V_{\epsilon_k}^*$ up to a subsequence, where $V_{\epsilon_k}^*$ is the $G$-replacement of $V_{\epsilon_k}$ in $\closure(B_s(G\cdot p))$. 
Combined with the constructions of the $G$-replacements $V_{\epsilon_k}^*$ (c.f. \cite{wang2023free}*{Proposition 5.3, 5.4, Lemma 5.5}) and the fact that $\Gamma=\bd\Omega\llcorner \interior(N)$ for some open $G$-set $\Omega\subset N$,
there exist sequences $\{\epsilon_j\}\subset \{\epsilon_k\}$ and $\{\tau_j\}_{j\in\mb N}\subset \mc Z^G_{n,rel}(N_{\epsilon_j},\bd N_{\epsilon_j};\mb Z_2) $ satisfying that $\lim_{j\to\infty}\epsilon_j = 0$,
\begin{enumerate}
    \item $\tau_j = [\partial \Omega_j]_{rel}$ is locally mass minimizing in $B_{s}(G \cdot p)$ with $\Omega_j \in \mathcal{C}^{G}(N_{\epsilon_j})$;     
    \item $|T_j| \rightarrow V^*= V_{\infty}$ in the sense of varifolds, where $T_j$ is the rectifiable mod $2$ $G$-invariant $n$-current induced by $\bd \Omega_j\llcorner \interior(N_{\epsilon_j})$.  
\end{enumerate}

Using the slicing theory \cite{simon1983lectures}*{28.5}, one can further choose $s\in (\frac{r}{2},r)$ and a small $G$-neighborhood $U\subset\subset \interior(N)$ of $\Gamma_i$ so that
\begin{itemize}
    \item $\spt(\|V_\infty\|)\cap U = \Gamma_i$ and $B_s(G\cdot p)\subset\subset U$;
    \item $\Omega_j\llcorner U \in\mc C^G(U)$ and $T_j \llcorner U=\bd ( \Omega_j\llcorner U)$ for all $j\in\mb N$;
    \item $T_j \llcorner B_{s}(G \cdot p)=\bd \Omega_j\llcorner B_s(G\cdot p)$ converges (up to a subsequence) to some $T_\infty\llcorner B_s(G\cdot p)=\bd \Omega_\infty\llcorner B_s(G\cdot p)\in \mf I^G_n(B_s(G\cdot p),\mb Z_2)$,
\end{itemize}
where $\Omega_\infty=\lim_{j\to\infty}\Omega_k\llcorner U\in \mc C^G(U)$ and $T_\infty=\bd \Omega_\infty\in \mf I^G_n(U;\mb Z_2)$. 
By the constancy theorem and the first bullet, we have $T_\infty=m_i'\llbracket \Gamma_i\rrbracket$ for some $m_i'\in\mb Z_2$. 
Since $T_\infty=\bd\Omega_\infty$, we know $\spt(\|T_\infty\|)$ admits a $G$-invariant unit normal, which implies $m_i'=0$ as $\Gamma_i$ is $(G,1)$-sided. 
By comparing $T_{\infty} \llcorner B_{s}(G \cdot p)$ with $V_{\infty} \llcorner B_s(G \cdot p)$ and adapting \cite{White09}*{Theorem 1.1}, we conclude that $m_i' \equiv m_i$ mod $2$, and thus $m_i$ is even whenever $\Gamma_i$ is $(G, 1)$-sided.  

\medskip
Now suppose that $\Gamma$ is not $G$-separating in $N$. 
Then, we only consider the case that   
$\Gamma=\bd \Omega\llcorner \interior (N)$ for some $\Omega\in\mc C(N)\setminus\mc C^G(N)$, i.e. $[\Gamma]_{rel}\in \mc B_{rel}(N,\bd N;\mb Z_2)\setminus\mc B^G_{rel}(N,\bd N;\mb Z_2)$ \footnote{\,The case that 
$[\Gamma]_{rel}\notin \mc B_{rel}(N,\bd N;\mb Z_2)$ is similar by Proposition \ref{Prop: lift group action}.}. By the proof of Proposition \ref{Prop: lift group action}(i), we have $\Omega\in \mc C^{G_\pm}(N)$ for an index $2$ Lie subgroup $G_+<G$ and $G_-=G\setminus G_+$.
We will argue similarly as in \cite{lwy24}*{Proposition 5.4} to get (ii). 
As 
$[\Gamma\llcorner N_{\epsilon}]_{rel}\in \mc B^{G_\pm}_{rel}(N_\epsilon,\bd N_\epsilon;\mb Z_2)$, for each small $\epsilon$ there exists a sequence 
$\tau^{\epsilon}_j = [\Gamma \llcorner N_{\epsilon}]_{rel} + [\partial R^{\epsilon}_j]_{rel} = [\partial \Omega^{\epsilon}_j]_{rel}\in \mc B^{G_\pm}_{rel}(N_{\epsilon},\bd N_{\epsilon};\mb Z_2)$, $j\in\mb N$, satisfying that 
\begin{enumerate}
    \item $R^\epsilon_j \in \mathcal{C}^{G}(N_{\epsilon})$ for all $j\in\mb N$; 
    \item $\Omega^{\epsilon}_j \in \mathcal{C}^{G_{\pm}}(N_{\epsilon})$ for all $j\in\mb N$;  
    \item $|T_j^{\epsilon}| \rightarrow V_{\epsilon}$ in the sense of varifolds, where $T_j^\epsilon$ is the rectifiable mod $2$ $G$-invariant $n$-current induced by $\bd \Omega^{\epsilon}_j\llcorner \interior(N_{\epsilon})$. 
\end{enumerate} 
Using the slicing theory \cite{simon1983lectures}*{28.5}, one can slightly shrink $N$ to a smooth $G$-domain $U\subset\subset \interior(N)$ containing $\spt(\|V_\infty\|)\cup \Gamma$ so that $R_j^{\epsilon_k}\llcorner U \in \mc C^G(U)$, $\Omega_j^{\epsilon_k}\llcorner U \in \mc C^{G_\pm}(U)$, and $T_j^{\epsilon_k}\llcorner U = \bd (\Omega_j^{\epsilon_k}\llcorner U)\in \mf I^G_n(U;\mb Z_2)$ for all $j,k\in\mb N$. By the compactness theorem, we can take a subsequence $\{\epsilon_{j}\}\subset \{\epsilon_k\}$ with $\lim_{j\to\infty}\epsilon_j=0$ so that $R_j^{\epsilon_j}\llcorner U\to R_\infty\in \mc C^G(U)$, $\Omega_j^{\epsilon_j}\llcorner U\to \Omega_\infty \in\mc C^{G_\pm}(U)$, and $T_j^{\epsilon_j}\llcorner U \to T_\infty=\bd\Omega_\infty=\llbracket \Gamma \rrbracket + \bd R_\infty\in \mc B^{G_\pm}(U;\mb Z_2)$ as $j\to\infty$. 
Combining the varifold convergence $\lim_{j\to\infty}|T_j^{\epsilon_j}\llcorner U|= V_\infty\llcorner U = V_\infty$ with the constancy theorem, one can conclude that  $T_\infty=\sum_{i=1}^{l_p}m_i'\llbracket \Gamma_i \rrbracket$ for some $\{m_i'\}\subset\mb Z_2$. 
Since $\Vol(\Omega_{\infty}) = \Vol(U)/2$, there exists a non-empty collection $ \cup_{i\in \mc I}\Gamma_i $ so that $0\neq \llbracket \cup_{i\in \mc I}\Gamma_i \rrbracket = \bd\Omega_\infty\in \mc B^{G_\pm}(U;\mb Z_2)$, and thus $[ \cup_{i\in \mc I}\Gamma_i ]_{rel} = [\partial\Omega_{\infty}]_{rel} = [\Gamma\llcorner U]_{rel} + [\partial R_{\infty}]_{rel}\in \mc B^{G_\pm}_{rel}(U,\bd U;\mb Z_2)$. 
Let $R_\infty'$ be the union of $R_\infty$ and the $G$-components of $N\setminus U$ that intersect $\closure(R_\infty)$. 
Hence, $[ \cup_{i\in \mc I}\Gamma_i ]_{rel}= [\Gamma \llcorner N]_{rel} + [\partial R_{\infty}']_{rel}\in \mc B^{G_\pm}_{rel}(N,\bd N;\mb Z_2)$, which implies $[ \cup_{i\in \mc I}\Gamma_i ]_{rel} \in [\tau]^G$. 
In particular, when $l_p = 1$, $\llbracket \cup_{i\in \mc I}\Gamma_i \rrbracket = \llbracket \Gamma_1 \rrbracket$ and $m_1'=1$. 
One can derive as before to show that $m_1\equiv m_1'$ mod $2$, and thus $m_1$ is odd.
\end{proof}

\section{Proof of Main Theorem}

Unless otherwise stated, all the $G$-hypersurfaces in this section are assumed to be closed and embedded. 

\subsection{Cutting constructions}

We first introduce the equivariant cutting constructions.  

\begin{lemma}\label{Lem: G-cutting}
	Given a compact Riemannian $G$-manifold $(N,g_{_N})$ with (possibly empty) boundary $\bd N$ satisfying \eqref{Eq: dimension assumption}, suppose $\bd N$ is minimal with a contracting $G$-neighborhood in $N$. 
	Let $\Gamma\subset \interior(N)$ be a $G$-connected closed embedded minimal $G$-hypersurface with a contracting or mixed $G$-neighborhood. 
	Then we can cut $N$ along some minimal $G$-hypersurfaces and get a different $G$-manifold $(N_1,g_{_{N_1}})$ so that 
	\begin{itemize}
		\item[(i)] $\bd N_1$ has a contracting $G$-neighborhood in $N_1$;
		\item[(ii)] $\interior(N_1)$ is $G$-equivariantly isometric to a domain in $N$;
		\item[(iii)] $\bd N_1$ has at least one boundary $G$-component that is isometric to $\Gamma$ (if $\Gamma$ is $(G,2)$-sided) or the $(G,2)$-sided double cover of $\Gamma$ (if $\Gamma$ is $(G,1)$-sided).
	\end{itemize}
	Moreover, suppose $\Gamma$ is not $G$-separating but has a mixed $G$-neighborhood, and $\Sigma\subset \interior(N)$ is a minimal $G$-hypersurface disjoint from $\Gamma$ with an expanding $G$-neighborhood. 
	Then after the cutting procedure and identifying $\interior(N_1)$ as a domain in $\interior(N)$, we can make (i)-(iii) hold and every $G$-connected component of $\Sigma$ lies either in $\interior(N_1)$ or in $\interior(N)\setminus\closure(\interior(N_1))$. 
\end{lemma}
\begin{proof}
	If $\Gamma$ is $(G,2)$-sided $G$-separating with a contracting or mixed $G$-neighborhood, then we can take the metric completion $(N_1,g_{_{N_1}})$ of a component of $N\setminus\Gamma$ so that $\bd N_1 = (\bd N\cap N_1) \sqcup \Gamma$ has a contracting $G$-neighborhood in $N_1$. The results in (i)-(iii) follow easily. 
	
	If $\Gamma$ is $(G,1)$-sided with a contracting $G$-neighborhood. We can simply take $N_1$ as the metric completion of $N_1\setminus\Gamma$. 
	Next, suppose $\Gamma$ is $(G,2)$-sided but not $G$-separating. 
    
	Let $(N',g_{_{N'}})$ be the metric completion of $N\setminus\Gamma$ so that $\bd N' = \bd N\sqcup \Gamma'\sqcup \Gamma''$, where $\Gamma'$ and $\Gamma''$ are both isometric to $\Gamma$. 
	If $\Gamma$ has a contracting $G$-neighborhood, then we can take $N_1:=N'$. 
	If $\Gamma$ has a mixed $G$-neighborhood, we may assume $\Gamma'$ and $\Gamma''$ have contracting and expanding $G$-neighborhoods in $N'$ respectively. 
    Suppose further that $\Sigma\subset \interior(N)$ is either an empty set or a minimal $G$-hypersurface disjoint from $\Gamma$ with an expanding $G$-neighborhood. 
    Thus, $\Sigma\subset \interior(N')$. 
    We are going to find a $G$-manifold satisfying (i)-(iii) and the last statement in the lemma. 

    Specifically, let $N''$ be the metric completion of $N'\setminus\Sigma$, which is a $G$-manifold with minimal boundary. 
    Since $N''$ could be $G$-disconnected, we consider the $G$-component $\wti N''$ of $N''$ that contains $\Gamma''$. 
    (Note that if $\Sigma=\emptyset$, then $\wti N''=N''=N'$.)
    Then by geometric measure theory and the maximum principle, we can minimize the area in the equivariant homology class $[\Gamma'']^G\in \mc H^G(\wti N'';\mb Z_2)$ (Remark \ref{Rem: minimizing in mfd with boundary}, \ref{Rem: locally G-boundary-type for G,2-sided}) to obtain a $(G,2)$-sided minimal $G$-hypersurface $\Gamma'''\subset \wti N''$ with a contracting $G$-neighborhood in $\wti N''$ so that 
    \begin{itemize}
        \item $\Gamma'''\sqcup\Gamma'' = \bd\Omega$ for some smooth open $G$-set $\Omega\subset \interior(\wti N'')$;
        \item every $G$-component of $\Gamma'''$ is contained either in $\bd \wti N''$ or in $\interior(\wti N'')$. 
    \end{itemize} 
    Note that every $G$-component of $\bd \wti N''$ comes from either $\Gamma'$, $\Gamma''$, a $G$-component of $\Sigma$, or a $G$-component of $\bd N$.  
	Therefore, since $\Area(\Gamma')=\Area(\Gamma'')>\Area(\Gamma''')$ and $\Sigma$ has an expanding $G$-neighborhood, we know that every $G$-component of $\Gamma'''\cap \bd \wti N''$ (if it exists) is isometric to a $G$-component of $\bd N$. 
    Additionally, one easily verifies that $\closure(\Omega)\cap \bd \wti N'' = \Gamma''\cup (\Gamma'''\cap \bd \wti N'')$. 
    Hence, after identifying $\Omega $ as a $G$-invariant domain in $\interior(N')$, we conclude that $\Gamma''\subset \bd \Omega$ and $(\Gamma'\cup\Sigma)\cap \Clos(\Omega)=\emptyset$. 
    Thus, by identifying $\Gamma''' $ as a closed minimal $G$-hypersurface in $N'$, we see the metric completion of $N'\setminus (\Gamma'''\cap\interior(N') )$ has $m\geq 2$ $G$-components $\{N_i'\}_{i=1}^m$ so that 
    \begin{itemize}
        \item $\Gamma''$ is contained in $N_m'\cong \closure(\Omega)$.
        \item $\Gamma'$ is contained in $N_1'$, and every $G$-component of $\Sigma$ is contained in some $N_i'$, $1\leq i\leq m-1$. 
    \end{itemize}
	The proof is finished by taking $N_1$ as the $G$-component $N_1'$ of $N'\setminus\Omega$ containing $\Gamma'$. 
\end{proof}

\begin{remark}\label{Rem: G-cutting procedure}
	In the above lemma, the results in (i) and (ii) are also valid without the $G$-connectedness assumption on $\Gamma\subset N$ if all the $G$-components of $\Gamma$ have the same type of $G$-neighborhood. 
	Specifically, if every $G$-component of $\Gamma$ has a mixed $G$-neighborhood and thus is $(G,2)$-sided, then by the above proof, we can consider the metric completion $N'$ of $N\setminus\Gamma$ (even if it may be not be $G$-connected), minimize the area in $[\Gamma'']^G\in H_n(N';\mb Z_2)$, and take $N_1$ to be the metric completion of any non-empty $G$-connected component of $N'\setminus \Omega$. 
	For the other cases, one can simply repeat using the above lemma on each $G$-component of $\Gamma$. 
\end{remark}

\subsection{Main theorems}

\begin{theorem}\label{Thm: infinite in trivial G-homology class}
	Let $(M^{n+1},g_{_M})$ be a closed Riemannian $G$-manifold satisfying \eqref{Eq: dimension assumption}, and $(N^{n+1},g_{_N})$ be a compact Riemannian $G$-manifold with (possibly empty) minimal boundary $\bd N$ so that $\bd N$ has a contracting $G$-neighborhood in $N$, and $\interior(N)$ is $G$-equivariantly isometric to a domain $M'\subset M$. 
	Suppose there are only finitely many $(G,1)$-sided minimal $G$-hypersurfaces in $M'\cong\interior(N)$. 
	Then there are infinitely many distinct minimal realizations $\Gamma$ of $[0]^G\in \mc H^G(M;\mb Z_2)$ supported in $\Clos(M')\subset M$. 
\end{theorem}
\begin{proof}
	Suppose for contradiction that there are only finitely many minimal realizations of $[0]^G\in \mc H^G(M;\mb Z_2)$ supported in $\Clos(M')$. 
	Then for any $G$-connected minimal $G$-hypersurface $\Gamma_1\subset \interior(N)$ with a non-expanding $G$-neighborhood, we can apply the cutting procedure in Lemma \ref{Lem: G-cutting} to have a new $G$-manifold $N_1$, which satisfies the same properties as $N$. 
    We next proceed with an inductive argument. 
    Suppose we have constructed $N_i$ for some $i\geq 1$ so that for each $1\leq j\leq i$, $\bd N_j$ is minimal with a contracting $G$-neighborhood in $N_j$, and $\interior(N_j)$ is isometric to a $G$-domain $M_j'\subset M_{j-1}'\subset M'$ in $M$, where $N_0:= N$ and $M_0':=M'$.  
    Denote by
	\begin{itemize}
		\item[(a)] $\mc N_i$ the set of all $G$-connected minimal $G$-hypersurfaces $\Gamma \subset \interior(N_i)\cong M_i'$ with a contracting or mixed $G$-neighborhood; 
		\item[(b)] $\mc B_i:=\{\beta=[\Gamma]^G\in \wti{\mc H}^G(M;\mb Z_2):  \Gamma \in \mc N_i \}$ (Definition \ref{Def: equivariant homology}). 
	\end{itemize}
    If $\mc N_i\neq \emptyset$ (or equivalently $\mc B_i\neq \emptyset$), 
    we can repeat the cutting procedure (Lemma \ref{Lem: G-cutting}) for any $\Gamma_i\in \mc N_i$ to have a new $G$-manifold $N_{i+1}$ so that $\bd N_{i+1}$ is minimal with a contracting $G$-neighborhood in $N_{i+1}$. 
	Moreover, $\interior(N_{i+1})$ is isometric to a $G$-domain $M_{i+1}'\subset M_i'$ in $ M$.  
	
    We claim that the above equivariant cutting procedure will stop in finite steps (say $m\in\mb N$ steps) so that $\mc N_m=\emptyset$ and $\mc B_m=\emptyset$. 
    Specifically, for any $j>i$, we know $\mc B_j\subset \mc B_i$ since $\interior( N_{j})$ is $G$-equivariantly isometric to a domain in $\interior(N_i)$. 
	Note that for any $(G,1)$-sided $G$-hypersurface $\Sigma\subset M$, any other realization $\Sigma'$ of $[\Sigma]^G\in \wti{\mc H}^G(M;\mb Z_2)$ must have a non-empty intersection with $\Sigma$ (Lemma \ref{Lem: intersect G1-sided}). 
	Hence, if $\Gamma_i\in \mc N_i$ is $(G,1)$-sided, then $[\Gamma_i]^G\notin \mc B_{j}$ for any $j>i$, which implies $\#\mc B_j < \#\mc B_i$ for $j>i$. 
	If $\Gamma_i\in \mc N_i$ is $(G,2)$-sided and $[\Gamma_i]^G\neq [0]^G\in \wti{\mc H}^G(M;\mb Z_2)$. 
    Then, since every minimal $G$-hypersurface $\Gamma\subset \interior (N_j)\cong M_j'$ for $j>i$ is disjoint from $\Gamma_i\subset M_i'\setminus M_j'$, we know that $\Gamma\cup\Gamma_i\subset M'$ is a minimal realization of $[0]^G$ in $M$ whenever $[\Gamma]^G=[\Gamma_i]^G\in \wti{\mc H}^G(M;\mb Z_2)$.
	Hence, by the contradiction assumption, there are only finitely many elements in $\mc N_j$ ($j>i$) that are $G$-homologous to $\Gamma_i$ in $M$. 
	Thus, after repeating for finite times, we have $[\Gamma_i]^G\notin \mc B_j$ and $\#\mc B_j < \#\mc B_i$ for some $j>i$. 
	If $\Gamma_i$ is $(G,2)$-sided and $[\Gamma_i]^G= [0]^G\in \wti{\mc H}^G(M;\mb Z_2)$, then our contradiction assumption indicates that $[0]^G$ can be eliminates from $\mc B_i$ in finite steps. 
	
	In summary, for any $\Gamma_i\in\mc N_i$, $[\Gamma_i]^G\in \wti {\mc H}^G(M;\mb Z_2)$ can be eliminated from $\mc B_i$ in finite steps. 
    By Lemma \ref{Lem: finite G-homology class}, Remark \ref{Rem: locally G-boundary-type for G,2-sided}, and the finiteness assumption on $(G,1)$-sided minimal $G$-hypersurfaces, we know $\mc B_1$ is finite. 
	Combined with $\mc B_j\subset \mc B_i$ for all $j>i$, we conclude that there exists $m\in\mb N$ so that $\mc N_m=\emptyset$. 
	Thus, every minimal $G$-hypersurface $\Gamma\subset \interior(N_m)\cong M_m'\subset M$ has an expanding $G$-neighborhood.  
	\begin{claim}\label{Claim: properties of core for [0]^G} 
		The `core' $G$-manifold $N_m$ also satisfies the following properties:
		\begin{itemize}
			\item[(1)] any two minimal $G$-hypersurfaces in $\interior(N_m)$ have a non-empty intersection; namely, $\interior(N_m)\cong M_m'$ satisfies the Frankel property for minimal $G$-hypersurfaces;
			\item[(2)] every $(G,2)$-sided minimal $G$-hypersurface in $\interior(N_m)$ is  $G$-separating in $N_m$; 
			\item[(3)] there are only finitely many minimal $G$-hypersurfaces in $\interior(N_m)$. 
		\end{itemize}
	\end{claim}
	\begin{proof}[Proof of Claim \ref{Claim: properties of core for [0]^G}]
        
        For the statement in (1), suppose $\Gamma,\Gamma'\subset \interior(N_m)\cong M_m'$ are two disjoint minimal $G$-hypersurfaces, which have expanding $G$-neighborhoods by the construction of $N_m$. 
        Then there exists a $G$-component $N_m'$ of $N_m\setminus (\Gamma\cup\Gamma')$ so that $N_m'$ has at least two new boundary $G$-components $S,S'$ coming from $\Gamma,\Gamma'$ respectively. 
        By minimizing the area in $[ S]^G \in \mc H^G(N_m';\mb Z_2)$ (Remark \ref{Rem: locally G-boundary-type for G,2-sided}), one can obtain a $(G,2)$-sided minimal $G$-hypersurface $S''$ in $\interior(N_m')$ with a contracting $G$-neighborhood by the maximum principle (cf. the proof of Lemma \ref{Lem: minimizer representative}, Remark \ref{Rem: minimizing in mfd with boundary}). 
        Since $\interior(N_m)\cong M_m'$, we get a contradiction to $\mc N_m=\emptyset$. 
        
		Combining $\mc N_m=\emptyset$ with the proof of Lemma \ref{Lem: G-cutting}, one easily obtains the statement in (2). 
        
		Since it has been assumed that $(G,1)$-sided minimal $G$-hypersurfaces in $M$ are finite, we only need to show that the set of $(G,2)$-sided minimal $G$-hypersurfaces in $\interior(N_m)$ is finite for (3). 
		If $\bd N_m=\emptyset$, i.e. $M=N=N_m$, then the finiteness result follows from (2) and our contradiction assumption. 
		If $\bd N_m\neq\emptyset$, then for any $(G,2)$-sided minimal $G$-hypersurface $\Sigma\subset \interior(N_m)$, there exists an open set $\Omega\subset \interior(N_m)\cong M_m'$ so that $\Sigma=\bd_{rel}\Omega:=\bd\Omega\cap \interior(N_m)$ by (2). 
		By identifying $\Omega$ as an open $G$-set in $M_m'\subset M$, we have a $G$-cobordant $\Clos(\Omega)$ between $\Sigma\subset M_m'= \interior(M_m')$ and a minimal $G$-hypersurface $\Gamma\subset \bd \Clos(M_m')$ that is isometric to part of $\bd N_m$. 
		Hence, any such $\Sigma$ is contained in a minimal realization $\Sigma\sqcup\Gamma$ of $[0]^G\in \mathcal{H}^G(M; \mathbb{Z}_2)$ supported in $\Clos(M')$, which implies (3) by the contradiction assumption. 
	\end{proof}
	
	If $\bd N_m=\emptyset$, i.e. $M=N=N_m$. 
	Then by the non-linear growth rate of the equivariant volume spectrum $\omega_p([0]^G )\sim p^{\frac{1}{l+1}}$ and the Frankel property (Claim \ref{Claim: properties of core for [0]^G}(1)), the proof of \cite{marques2017existence}*{Theorem 6.1 and 1.1} can be taken almost verbatim to show the existence of infinitely many $G$-connected minimal $G$-hypersurfaces in $\interior(N_m)=M$. 
	This contradicts Claim \ref{Claim: properties of core for [0]^G}(3). 
	
	If $\bd N_m\neq\emptyset$, let $\{\widehat\Sigma_i\}_{i=1}^q$ be the $G$-connected components of $\bd N_m$. 
    One can easily adapt the proof of \cite{song2023infinite}*{Lemma 13} (using Lemma \ref{Lem: minimizer representative} and Remark \ref{Rem: minimizing in mfd with boundary}) to show that any $G$-connected minimal $G$-hypersurface $\Sigma\subset \interior(N_m)\cong M_m'$ satisfies that  
	\begin{itemize}
		\item $\Area(\Sigma) > \sup_{i\in\{1,\dots,q\}} \Area(\widehat\Sigma_i)$ if $\Sigma$ is $(G,2)$-sided;
		\item $2\Area(\Sigma) > \sup_{i\in\{1,\dots,q\}} \Area(\widehat\Sigma_i)$ if $\Sigma$ is $(G,1)$-sided. 
	\end{itemize}
	Consider the non-compact $G$-manifold $Cyl(N_m)$ with cylindrical ends associated to $N_m$. 
	It follows from Claim \ref{Claim: properties of core for [0]^G}(1) and Theorem \ref{Thm: confined min-max hypersurface}(i) w.r.t. $[0]_c^G$ in $Cyl(N_m)$ that $\omega_p([0]_c^G)$ is realized by the area of a $G$-connected minimal $G$-hypersurface $\wti\Sigma_p\subset M_m'\cong \interior(N_m)$ with integer multiplicity 
	\[ \omega_p([0]_c^G)=m_p \Area(\wti\Sigma_p), \]
	where $p,m_p\in\mb N$, and $m_p$ is even whenever $\wti \Sigma_p$ is $(G,1)$-sided. 
	Together with Claim \ref{Claim: properties of core for [0]^G}(3) and Theorem \ref{Thm: width estimates in cylindrical mfd}, we have a contradiction to the arithmetic result \cite{song2023infinite}*{Lemma 14}.
\end{proof}

We are now ready to show our main theorem. 

\begin{theorem}\label{Thm: infinite in G-homology class}
	Let $(M^{n+1},g_{_M})$ be a closed Riemannian manifold with a compact Lie group $G$ acting by isometries so that \eqref{Eq: dimension assumption} is satisfied. 
	Then 
	either
	\begin{itemize}
		\item[(i)] there are infinitely many embedded, $G$-connected, $(G,1)$-sided minimal hypersurfaces; or
		\item[(ii)] for any $\alpha\in \mc H^G(M;\mb Z_2)$, there are infinitely many minimal realizations of $\alpha$. 
	\end{itemize}
    In particular, there always exist infinitely many $G$-invariant minimal hypersurfaces in $M$.
\end{theorem}

\begin{proof}
	Suppose for a contradiction that there are only finitely many $(G,1)$-sided minimal $G$-hypersurfaces and finitely many minimal realizations of certain $\alpha\in \mc H^G(M;\mb Z_2)$. 
	By Theorem \ref{Thm: infinite in trivial G-homology class}, $\alpha\neq [0]^G$. 
	
	\medskip
	\noindent {\bf Step 1.} {\it Cut along the area minimizer of $\alpha$.}
	
	Let $\Gamma_{0}$ be the area minimizing essential realization of $\alpha$ with $G$-connected components $\{\Gamma_{0}^{(i)}\}_{i\in\mc I}$ for some finite set $\mc I \subset \mb N$ (Lemma \ref{Lem: minimizer representative}). 
	Since $\Gamma_0$ is an essential realization of $\alpha\neq 0$, each $\Gamma_0^{(i)}$ is not $G$-separating and cannot be $(G,2)$-sided with a mixed or expanding $G$-neighborhood. 
	Hence, we can divide $\mc I = \mc I_0 \sqcup \mc I_1 \sqcup \mc I_2$ so that 
	\begin{align*}
		\Gamma_0^{(i)} \mbox{ is } 
		\left\{
		\begin{array}{ll}
			\mbox{$(G,1)$-sided with an expanding $G$-neighborhood} & \mbox{for $i\in \mc I_0$, }\\
			\mbox{$(G,1)$-sided with a contracting $G$-neighborhood} & \mbox{for $i\in \mc I_1$, }\\
			\mbox{$(G,2)$-sided non-$G$-separating with a contracting $G$-neighborhood} & \mbox{for $i\in \mc I_2$. }
		\end{array}
		\right.
	\end{align*}
    
	Then by Lemma \ref{Lem: G-cutting}, we can successively cut $M$ along each $\Gamma_0^{(i)}$ for $i\in \mc I_1\cup\mc I_2$ to obtain a new $G$-manifold $N_0$ so that $\bd N_0$ is minimal with a contracting $G$-neighborhood. 
	By identifying $\interior(N_0)$ with a $G$-domain $M_0$ in $M$, we denote by
	\begin{itemize}
		\item $\Gamma_0^e:= \cup_{i\in\mc I_0}\Gamma_0^{(i)} \subset \interior(N_0)$ a minimal $G$-hypersurface with an expanding $G$-neighborhood; 
		\item $\alpha_0:= [\cup_{i\in\mc I_0}\Gamma_0^{(i)}]^G \in \mc H^G(M;\mb Z_2)$, and 
        $\wti\alpha_0:=[\cup_{i\in\mc I_0}\Gamma_0^{(i)}]_{rel} + \mc B^G_{rel}(N_0,\bd N_0;\mb Z_2)$ the relative $G$-homology class of $\Gamma^e_0$ in $\mc {LB}^G_{rel}(N_0,\bd N_0;\mb Z_2)$;
		\item $\Sigma_0:= \cup_{i\in \mc I_1\cup\mc I_2} \Gamma_0^{(i)}$ an embedded minimal $G$-hypersurface in $M$ with $\alpha = \alpha_0 + [\Sigma_0]^G$ and $\Sigma_0\cap M_0=\emptyset$;
		\item $\mc W_0:=\{\Gamma_0\}$ so that every $\Gamma\in\mc W_0$ is a minimal realization of $\alpha$ and every $G$-component of $\Gamma$ contained in $M_0\cong\interior(N_0)$ has an expanding $G$-neighborhood. 
	\end{itemize}
    Let $\mc E_0$ be the collection of closed minimal realizations of $\wti\alpha_0$ that are compactly supported in $\interior(N_0)$. 
    Then we make the following claim:
    \begin{claim}\label{Claim: finite E}
        For any $\Gamma\in \mc E_0$, there exists an embedded minimal $G$-hypersurface $\Gamma'\subset \bd M_0$ so that $\Gamma'$ is isometric to the union of some $G$-components of $\bd N_0$, and $\Gamma\sqcup (\Sigma_0\triangle\Gamma')$ is a minimal realization of $\alpha$, where $\Sigma_0\triangle\Gamma':=(\Sigma_0\setminus \Gamma') \cup (\Gamma'\setminus\Sigma_0)$. 
        In particular, $\mc E_0$ is finite by assumptions. 
    \end{claim}
    \begin{proof}[Proof of Claim \ref{Claim: finite E}]
        Let $\Gamma\subset \interior(N_0)\cong M_0$ be a closed minimal realization of $\wti\alpha_0$. 
        Namely, there exists an open $G$-set $\Omega\subset \interior (N_0)$ so that $\llbracket \Gamma^e_0 \rrbracket + \llbracket\Gamma\rrbracket = (\bd \Omega) \llcorner \interior(N_0) $ in $N_0$.
        Since $\Gamma,\Gamma^e_0\subset \interior(N_0)$ are closed $G$-hypersurfaces,  $\bd\Omega\llcorner \bd N_0$ is induced by the union of some $G$-components of $\bd N_0$. 
        After identifying $\Omega$ as an open $G$-set in $M_0$ and taking the closure $\closure(\Omega)$ in $M$, we have $\bd \closure(\Omega)= \llbracket \Gamma^e_0 \rrbracket + \llbracket \Gamma\sqcup \Gamma' \rrbracket$ in $M$, where $\Gamma'= \bd M_0\cap \closure(\Omega)$ is a (possibly empty) minimal $G$-hypersurface isometric to the union of some $G$-components of $\bd N_0$. 
        Hence, $\Gamma\sqcup \Gamma'\subset M$ is also a minimal realization of $\alpha_0$. Note that $\Gamma\subset \interior (M_0)$, $\Sigma_0\subset M\setminus M_0$, and $\Gamma'\subset \bd M_0$ is the union of some minimal $G$-hypersurfaces we have cut along.
        Thus, $\Sigma_0\triangle \Gamma'$ is an embedded minimal $G$-hypersurface that is disjoint to $\Gamma$. 
        In particular, $\Sigma_0\triangle (\Gamma \sqcup \Gamma')= \Gamma \sqcup (\Sigma_0\triangle \Gamma')$ gives a minimal realization of $[\Sigma_0]^G+\alpha_0=\alpha$ in $M$ which coincides with $\Gamma$ when restricted in $M_0$. 
        This implies that $\# \mc E_0$ is bounded from above by the number of minimal realizations of $\alpha$. 
    \end{proof}

    Note that every $\Gamma\in \mc E_0$ is a local $G$-boundary. 
	Let $\mc E_0^e$ be the set of $\Gamma\in \mc E_0$ with an expanding $G$-neighborhood.  
	In particular, $\Gamma_0^e\in \mc E_0^e$ if they are non-empty. 
	
	To proceed, if $\alpha_0=[0]^G$, then by Theorem \ref{Thm: infinite in trivial G-homology class}, there are infinitely many minimal realizations $\Gamma$ of $[0]^G\in \mc H^G(M;\mb Z_2)$ supported in $\Clos(M_0)$. 
	Hence, $\alpha=[\Sigma_0]^G+[0]^G=[\Sigma_0]^G+[\Gamma]^G=[\Sigma_0\triangle \Gamma]^G$. 
	This yields infinitely many minimal realizations $\Sigma_0\triangle \Gamma$ of $\alpha$, which is a  contradiction. 
	
	Suppose next that $\alpha_0\neq [0]^G$, so $\Gamma_0^e\neq \emptyset$ and $\emptyset\neq \mc E_0^e\subset \mc E_0$. 
	If $\mc E_0=\mc E_0^e$ then we change all the subscripts from $_0$ to $_1$ and move on to {\bf Step 3}. Otherwise, continue with {\bf Step 2}. 
	
	\medskip
	\noindent{\bf Step 2.} {\it Cut along non-expanding realizations.}
	
	Given $\Gamma_{0,1}\in \mc E_0\setminus \mc E_0^e$ with $G$-components $\cup_{i\in\mc I}\Gamma_{0,1}^{(i)}$, let $\Gamma_{0,1}'\subset \bd M_0$ be the minimal $G$-hypersurface associated to $\Gamma_{0,1}$ from Claim \ref{Claim: finite E}.
    Since $\Gamma_{0,1}$ does not have an expanding $G$-neighborhood, we have 
	\[\Gamma_{0,1} \sqcup (\Sigma_{0} \triangle \Gamma_{0,1}' )\in \alpha \quad {\rm and}\quad  \Gamma_{0,1} \sqcup (\Sigma_{0} \triangle \Gamma_{0,1}' ) \notin \mc W_0. \]
	Write $\mc I=\mc I_0\sqcup \mc I_1\sqcup \mc I_2 \sqcup \mc I_3$ so that 
	\begin{align*}
		\Gamma_{0,1}^{(i)} 
		\left\{
		\begin{array}{ll}
			\mbox{has an expanding $G$-neighborhood} & \mbox{for $i\in \mc I_0$, }\\
			\mbox{is non-$G$-separating in $N_0$ with a contracting $G$-neighborhood} & \mbox{for $i\in \mc I_1$,}\\
			\mbox{is $G$-separating in $N_0$ with a contracting or mixed $G$-neighborhood} & \mbox{for $i\in \mc I_2$,}\\
			\mbox{is $(G,2)$-sided non-$G$-separating in $N_0$ with a mixed $G$-neighborhood} & \mbox{for $i\in \mc I_3$.}\\
		\end{array}
		\right.
	\end{align*}
	The cutting procedure (Lemma \ref{Lem: G-cutting}) along $\Gamma_{0,1}^{(i)}$ for $i \in \mc I_1$ will only add boundary components, and will chop off $G$-domains bounded by  $\Gamma_{0,1}^{(i)}$ for $i \in \mc I_2$. 
	Hence, by Lemma \ref{Lem: G-cutting}, we can successively cut $N_0$ along each $\Gamma_{0,1}^{(i)}$ for $i\in \mc I_1\cup\mc I_2$ to obtain a new $G$-manifold $N_{0}'$ so that $\bd N_0'$ is minimal with a contracting $G$-neighborhood. 
	After identifying $\interior(N_0')$ as a domain $M_0'$ in $M_0\subset M$, take $\mc J'\subset \mc I_0\cup\mc I_3$ as the set of indices $j$ so that $\Gamma_{0,1}^{(j)} \subset M_0'\cong\interior (N_0')$.  
	
	Next, by applying Remark \ref{Rem: G-cutting procedure} and the last statement in Lemma \ref{Lem: G-cutting} to $ \cup_{j\in\mc J'\cap \mc I_3} \Gamma_{0,1}^{(j)}$ and $\cup_{j\in\mc J'\cap \mc I_0} \Gamma_{0,1}^{(j)}$ in $N_0'$ (in place of $\Gamma,\Sigma$ respectively), we obtain a new $G$-manifold $N_{0,1}$ so that $\bd N_{0,1}$ is minimal with a contracting $G$-neighborhood, $\interior(N_{0,1})$ can be identified with a $G$-domain $M_{0,1}$ in $M_0'\subset M$, and every $\Gamma_{0,1}^{(j)}$, $j\in \mc J'\cap \mc I_0$, is either contained in $M_{0,1}$ or contained in $ M_0'\setminus \closure(M_{0,1}) $.  
	Let $\mc J\subset \mc I_0$ be the set of indices $j$ with $\Gamma_{0,1}^{(j)} \subset M_{0,1}\cong\interior (N_{0,1})$.  Define 
	\begin{itemize}
		\item $\Gamma_{0,1}^e:= \cup_{j\in \mc J} \Gamma_{0,1}^{(j)}\subset \interior(N_{0,1})$ as a minimal $G$-hypersurface with an expanding $G$-neighborhood; 
		\item $\alpha_{0,1}:=[ \cup_{j\in \mc J} \Gamma_{0,1}^{(j)}]^G\in \mc H^G(M;\mb Z_2)$, and 
        $\wti\alpha_{0,1}:=[ \cup_{j\in \mc J} \Gamma_{0,1}^{(j)}]_{rel} + \mc B^G_{rel}(N_{0,1},\bd N_{0,1};\mb Z_2)$ as the relative $G$-homology class of $\Gamma_{0,1}^e$ in $\mc {LB}^G_{rel}(N_{0,1}, \bd N_{0,1};\mb Z_2)$;
		\item $\Sigma_{0,1}:= \Sigma_0 \triangle \Gamma_{0,1}' \cup (\cup_{i\in \mc I \setminus\mc J } \Gamma_{0,1}^{(i)} )$ as an embedded minimal $G$-hypersurface in $M$ with $\alpha = \alpha_{0,1} + [\Sigma_{0,1}]^G$ and $\Sigma_{0,1}\cap M_{0,1}=\emptyset$;
		\item $\mc W_{0,1}:= \mc W_0 \cup\{\Gamma_{0,1}\cup (\Sigma_0\triangle \Gamma_{0,1}')= \Gamma_{0,1}^e\cup \Sigma_{0,1}\}$ so that every $\Gamma\in\mc W_{0,1}$ is a minimal realization of $\alpha$ and every $G$-component of $\Gamma$ contained in $M_{0,1}\cong\interior(N_{0,1})$ has an expanding $G$-neighborhood. 
	\end{itemize}
    By the same arguments in Claim \ref{Claim: finite E}, we can define
	\begin{itemize}
		\item $\mc E_{0,1}$ as the {\em finite} set of closed minimal realizations of $\wti\alpha_{0,1}$ supported in $\interior(N_0)$;
		\item $\mc E_{0,1}^e$ as the set of $\Gamma\in\mc E_{0,1}$ with an expanding $G$-neighborhood;
        \item $\Gamma'\subset \bd M_{0,1}$ as the minimal $G$-hypersurface associated to $\Gamma\in \mc E_{0,1}$ so that $\Gamma'$ is isometric to the union of some $G$-components of $\bd N_{0,1}$, and $\Gamma \sqcup (\Sigma_{0,1}\triangle \Gamma')$ is a minimal realization of $\alpha$. 
	\end{itemize}
	In particular, $\Gamma_{0,1}^e\in \mc E_{0,1}^e$ if they are not empty.

	If $\alpha_{0,1}\neq [0]^G$ (so $\emptyset\neq \Gamma_{0,1}^e\in\mc E_{0,1}^e$), and $\mc E_{0,1}^e\subsetneq\mc E_{0,1}$, 
	then we repeat the above constructions for any $\Gamma_{0,2}\in \mc E_{0,1}\setminus\mc E_{0,1}^e$ with all the subscripts $_{0}$ and $_{0,1}$ replaced by $_{0,1}$ and $_{0,2}$ respectively. 
	In particular, $\Gamma_{0,2}\sqcup (\Sigma_{0,1} \triangle \Gamma_{0,2}')$ is a minimal realization of $\alpha$, but it does not belong to $\mc W_{0,1}$ since $\Gamma_{0,2}$ does not admit an expanding $G$-neighborhood in $\interior(N_{1,0})$. 
	Hence, $\mc W_{0,2}$ will have one more element $\Gamma_{0,2}\sqcup (\Sigma_{0,1} \triangle \Gamma_{0,2}')$ than $\mc W_{0,1}$. 
	Together with our finiteness assumption, this procedure must stop in a finite number (say $m_0$) of times with either $\mc E_{0,m_0}^e=\mc E_{0,m_0}\neq \emptyset$ or $\alpha_{0,m_0}=[0]^G$. 
	In the latter case, Theorem \ref{Thm: infinite in trivial G-homology class} gives a contradiction as in {\bf Step 1}.  
	Suppose now $\alpha_{0,m_0}\neq [0]^G$ and  $\mc E_{0,m_0}=\mc E_{0,m_0}^e\neq \emptyset$, then we change all the subscripts $_{0,m_0}$ to $_1$ and move on to {\bf Step 3}.

	\medskip
	\noindent{\bf Step 3.} 
	Now, we have the following objects:
	\begin{itemize}
		\item[(1)] a compact $G$-manifold $N_1$ so that $\bd N_1$ is minimal with a contracting $G$-neighborhood, and $\interior(N_1)$ is $G$-equivariantly isometric to a $G$-domain $M_1$ in $M_0\subset M$; 
		\item[(2)] $\Gamma_{1}^e\subset \interior(N_{1})$ is a (non-trivial) minimal $G$-hypersurface of locally $G$-boundary-type with an expanding $G$-neighborhood; 
		\item[(3)] $\alpha_{1}=[ \Gamma_{1}^e]^G\neq [0]^G\in \mc H^G(M;\mb Z_2)$, and 
        $\wti\alpha_{1}=[ \Gamma_{1}^e]_{rel} + \mc B^G_{rel}(N_1, \bd N_1;\mb Z_2)$ is the relative $G$-homology class of $\Gamma_1^e$ in $\mc {LB}^G_{rel}(N_1,\bd N_1;\mb Z_2)$;
		\item[(4)] $\Sigma_{1}$ is an embedded minimal $G$-hypersurface in $M$ with $\alpha = \alpha_{1} + [\Sigma_{1}]^G$ and $\Sigma_{1}\cap M_1=\emptyset$;
		\item[(5)] $\mc E_{1}\neq \emptyset$ is the finite set of closed minimal realizations of $\wti\alpha_{1}$ supported in $\interior(N_1)$ so that every element of $\mc E_{1}$ has an expanding $G$-neighborhood; 
        \item[(6)] $\Gamma'\subset \bd M_1$ is a minimal $G$-hypersurface associated to $\Gamma\in \mc E_{1}$ so that $\Gamma'$ is isometric to the union of some $G$-components of $\bd N_{1}$, and $\Gamma \sqcup (\Sigma_{1}\triangle \Gamma')$ is a minimal realization of $\alpha$. 
	\end{itemize}

	\medskip
	\noindent{\bf Step 4. }{\it Cut along minimal $G$-hypersurfaces with a non-expanding $G$-neighborhood that are disjoint from some $\Gamma\in\mc E_1$.}
	
	Fix any $\Gamma_1 \in \mc E_1$ with $G$-components $\{\Gamma_1^{(i)}\}_{i\in \mc I}$. Let $\Gamma_1'\subset \bd M_1$ be the associated minimal $G$-hypersurface given by (6) in {\bf Step 3}.

    \medskip
    \noindent{\bf Step 4.a.}
    We apply the following cutting procedure with subscripts $_{1,i}$ along minimal $G$-hypersurfaces that are disjoint from $\Gamma_1$ with non-expanding $G$-neighborhoods. 
    Specifically,
	let 
	\begin{itemize}
		\item[(a)] $\mc N_1$ be the set of all $G$-connected minimal $G$-hypersurfaces $\Gamma \subset \interior(N_1)$ with a non-expanding $G$-neighborhood so that $\Gamma\cap \Gamma_1=\emptyset$; 
		\item[(b)] $\mc B_1:=\{\beta=[\Gamma]^G\in \wti{\mc H}^G(M;\mb Z_2):  \Gamma \in \mc N_1 \}$ (c.f. Definition \ref{Def: equivariant homology}). 
	\end{itemize}
	Suppose $\mc N_1\neq\emptyset$ (equivalently $\mc B_1\neq\emptyset$). Then, take any $\beta\in\mc B_1$ and $\Gamma\in\mc N_1$ with $\beta=[\Gamma]^G$. 
	
	{\bf Case 1.} $\Gamma \in \mc N_{1}$ is $(G,1)$-sided. 
	Then $\Gamma$ is disjoint from $\Gamma_1$ and has a contracting $G$-neighborhood.  
	Therefore, we can consider the metric completion $N_{1,1}$ of $N_1\setminus \Gamma$ as in Lemma \ref{Lem: G-cutting}, and define $M_{1,1}:=M_1\setminus\Gamma$, $\Gamma_{1,1}^e:=\Gamma_1\subset\interior(N_{1,1})$, $\alpha_{1,1}:=[\Gamma_{1,1}^e]^G\in \mc H^G(M;\mb Z_2)$, $\wti\alpha_{1,1}:=[\Gamma_{1,1}^e]_{rel}+ \mc B^G_{rel}(N_{1,1}, \bd N_{1,1};\mb Z_2)$, 
    $\Sigma_{1,1}:=\Sigma_1\triangle \Gamma_1'$, 
    so that (1)-(4) in {\bf Step 3} remain valid for the objects with subscripts $_{1,1}$. 
	Let $\mc E_{1,1}$ be defined as in {\bf Step 3} (5) with $\wti\alpha_{1,1},N_{1,1}$ in place of $\wti\alpha_1,N_1$ respectively. 
    Then the proof of Claim \ref{Claim: finite E} would carry over to show that $\mc E_{1,1}$ is finite, and every element of $\mc E_{1,1}$ has an associated minimal $G$-hypersurface satisfying the statement in {\bf Step 3} (6).
    Moreover, since $\Gamma$ is $(G,1)$-sided, one can identify $\interior(N_{1,1})$ as a $G$-domain in $N_1$ and verify that 
    \begin{align}\label{Eq: E_11 subset E_1}
        \mc E_{1,1}\subset \mc E_1,
    \end{align} 
    which implies that every element of $\mc E_{1,1}$ has an expanding $G$-neighborhood. 
    In particular, the statement in {\bf Step 3} (5) remains valid.
	We can also update $\mc N_{1,1}$ and $\mc B_{1,1}$ in the new $G$-manifold $N_{1,1}$ as in (a)(b) with $\Gamma_{1,1}^e$ in place of $\Gamma_1$. 
	Note that every realization $\widehat \Gamma\subset M$ of $\beta$ must intersect with the $(G,1)$-sided realization $\Gamma$ (Lemma \ref{Lem: intersect G1-sided}). 
	Hence, there is no minimal realization of $\beta$ in $\mc N_{1,1}$, which implies 
    \begin{align}\label{Eq: decrease B_1}
        \beta\notin\mc B_{1,1}\subsetneq \mc B_1. 
    \end{align}
	
	{\bf Case 2.} $\Gamma\in \mc N_{1}$ is $(G,2)$-sided non-$G$-separating in $N_1$ with a contracting $G$-neighborhood. 
	 Then, let $N_{1,1}$ be the metric completion of $N_{1}\setminus \Gamma$ as in Lemma \ref{Lem: G-cutting}, and take $M_{1,1}:=M_1\setminus\Gamma$, $\Gamma_{1,1}^e:=\Gamma_1\subset\interior(N_{1,1})$, $\alpha_{1,1}:=[\Gamma_{1,1}^e]^G\in \mc H^G(M;\mb Z_2)$, $\wti\alpha_{1,1}:=[\Gamma_{1,1}^e]_{rel}+ \mc B^G_{rel}(N_{1,1}, \bd N_{1,1};\mb Z_2)$,
     $\Sigma_{1,1}:=\Sigma_1\triangle \Gamma_1'$, 
     satisfying (1)-(4) in {\bf Step 3}. 
	 Also, 
	 we update $\mc E_{1,1}$, $\mc N_{1,1}$ and $\mc B_{1,1}$ similarly in the new $G$-manifold $N_{1,1}$ with $\wti\alpha_{1,1},\Gamma_{1,1}^e$ in place of $\wti\alpha_{1},\Gamma_{1}$ respectively. 
     By the proof of Claim \ref{Claim: finite E}, $\mc E_{1,1}$ is finite and the statement in {\bf Step 3} (6) remains valid for $\mc E_{1,1}$. 
     
     Additionally, if $\wti \Gamma \in \mc E_{1,1}$, then there exists an open $G$-set $\Omega\subset \interior (N_{1,1})$ so that $\llbracket \wti\Gamma \rrbracket + \llbracket \Gamma_{1,1}^e \rrbracket = \bd \Omega \llcorner \interior(N_{1,1}) $. 
     By identifying $\Omega$ as an open $G$-set in $\interior(N_1)$ and taking the closure $\closure(\Omega)$ in $N_1$, we see from the constructions that $\bd \closure(\Omega) \llcorner \interior (N_1)$ is either $\llbracket \Gamma_1\rrbracket + \llbracket \wti \Gamma \rrbracket$ or $\llbracket \Gamma_1\rrbracket + \llbracket \wti \Gamma \sqcup \Gamma \rrbracket$. 
     However, the latter case cannot occur since otherwise, $\wti \Gamma \sqcup \Gamma \in \mc E_1$ has a non-expanding $G$-neighborhood, which contradicts {\bf Step 3} (5).
     Therefore, by identifying $\interior(N_{1,1})$ as a $G$-domain in $N_1$, we see $\wti\Gamma\in \mc E_1$. 
     Thus, \eqref{Eq: E_11 subset E_1} remains valid, and every $\wti\Gamma\in\mc E_{1,1}$ is has an expanding $G$-neighborhood. 
     This verified {\bf Step 3} (5) for $\mc E_{1,1}$. 
     
     Moreover, one can also easily verify by the constructions that
     \begin{align}
        \mc B_{1,1}\subset \mc B_1,\qquad{\rm and}\qquad
         \beta\cap\mc N_{1,1} \subsetneq \beta\cap \mc N_1 \quad{\rm since~}\Gamma\notin\mc N_{1,1} .
     \end{align}
     Furthermore, if there exists $\widehat\Gamma\in \beta\cap\mc N_{1,1}$, 
	 then $\Gamma\cup \widehat\Gamma$ is a minimal realization of $[0]^G\in \mc H^G(M;\mb Z_2)$ and is disjoint from $\Gamma_{1,1}^e\cup\Sigma_{1,1}\in \alpha$. 
	 Thus, $(\Gamma\cup \widehat\Gamma)\cup (\Gamma_{1,1}^e\cup\Sigma_{1,1})$ is a new minimal realization of $\alpha$. 
	 By the assumption, there are only finitely many minimal realizations of $\beta$ in $\mc N_{1,1}$, i.e. $\#\beta\cap\mc N_{1,1}<\infty$. 

	{\bf Case 3.} $\Gamma \in \mc N_{1}$ is $(G,2)$-sided non-$G$-separating in $N_1$ with a mixed $G$-neighborhood. 
	Then by the last statement of Lemma \ref{Lem: G-cutting} with $\Gamma, \Gamma_1$ in place of $\Gamma,\Sigma$ respectively, we can cut $N_1$ along $\Gamma$ and some minimal $G$-hypersurfaces in $\interior(N_1)\setminus(\Gamma\cup\Gamma_1)$ to obtain a new $G$-manifold $N_{1,1}$ so that $\interior(N_{1,1})$ can be identified as a $G$-domain $M_{1,1}$ in $M_1\subset M$, and $\bd N_{1,1}$ is minimal with a contracting $G$-neighborhood. 
	Additionally, $\bd N_{1,1}$ has a $G$-component isometric to $\Gamma$, and each $\Gamma_1^{(i)}$ is contained either in $\interior(N_{1,1})$ or in $\interior( N_1)\setminus \closure( \interior(N_{1,1}))$. 
	Hence, let $\mc J\subset \mc I$ be the set of indices $j$ with $\Gamma_{1}^{(j)} \subset \interior (N_{1,1})$.  
    Denote by
	\begin{itemize}
		\item $\Gamma_{1,1}^e:= \cup_{j\in \mc J} \Gamma_{1}^{(j)}\subset \interior(N_{1,1})$ a minimal $G$-hypersurface of locally $G$-boundary-type with an expanding $G$-neighborhood; 
		\item $\alpha_{1,1}:=[ \cup_{j\in \mc J} \Gamma_{1,1}^{(j)}]^G\in \mc H^G(M;\mb Z_2)$, and $\wti\alpha_{1,1}:=[ \cup_{j\in \mc J} \Gamma_{1,1}^{(j)}]_{rel} + \mc B^G_{rel}(N_{1,1}, \bd N_{1,1};\mb Z_2)$;
		\item $\Sigma_{1,1}:=  (\Sigma_{1}\triangle \Gamma_1') \cup (\cup_{i\in \mc I \setminus\mc J } \Gamma_{1}^{(i)} )$ an embedded minimal $G$-hypersurface in $M$ with $\alpha = \alpha_{1,1} + [\Sigma_{1,1}]^G$ and $\Sigma_{1,1}\cap M_{1,1}=\emptyset$;
		\item $\mc E_{1,1}$ is the set of closed minimal realizations of $\wti\alpha_{1,1}$ supported in $\interior(N_{1,1})$.
	\end{itemize}
	If $\alpha_{1,1}\neq [0]^G$, then $\emptyset\neq\Gamma_{1,1}^e\subset \Gamma_1$ by the constructions. 
    Additionally, similar to {\bf Case 2}, every $\wti\Gamma \in \mc E_{1,1}$ can be unionized with $\Gamma_1\setminus\Gamma_{1,1}^e$ and some (possibly empty) minimal $G$-hypersurfaces (that we cut along in this step) to give an element in $\mc E_1$.
    Recall that we are cutting along minimal $G$-hypersurfaces with non-expanding $G$-neighborhoods, and elements in $\mc E_1$ all have expanding $G$-neighborhoods. 
    Hence, $\wti\Gamma\sqcup (\Gamma_1\setminus\Gamma_{1,1}^e) \in \mc E_{1,1}$. 
	In particular, $\mc E_{1,1}$ not only satisfies (5) and (6) in {\bf Step 3} with $\wti\alpha_{1,1},N_{1,1}$ in place of $\wti\alpha_1,N_1$ respectively, but also satisfies
    \[\#\mc E_{1,1}\leq \#\mc E_1. \]
	Define $\mc N_{1,1}$ and $\mc B_{1,1}$ as in (a)(b) with $\Gamma_{1,1}^e,N_{1,1}$ in place of $\Gamma_1,N_1$ respectively. 
	Similar to the above {\bf Case 2}, we also have $\mc B_{1,1}\subset\mc B_1$, $\beta\cap\mc N_{1,1}\subsetneq \beta\cap \mc N_1$, and $\#\beta\cap\mc N_{1,1}<\infty$. 

	{\bf Case 4.} $\Gamma\in \mc N_{1}$ is $G$-separating in $N_{1}$ with a contracting or mixed $G$-neighborhood. 
	By Lemma \ref{Lem: G-cutting}, let $N_{1,1}$ be a $G$-component of $N_{1}\setminus\Gamma$ so that $\bd N_{1,1}$ has a contracting $G$-neighborhood. 
	Then, we can similarly define $\Gamma_{1,1}^e$, $\alpha_{1,1}$, $\wti\alpha_{1,1}$, $\Sigma_{1,1}$, $\mc E_{1,1}$, $\mc N_{1,1}$, and $\mc B_{1,1}$ as in {\bf Case 3}. 
	Additionally, we know that either $ \alpha_{1,1}=[0]^G$, or every element in $\mc E_{1,1}\neq \emptyset$ has an expanding $G$-neighborhood and the statements in {\bf Step 3} (1)-(6) are satisfied with subscripts $_{1,1}$ in place of $_1$. 
	In the latter case, one can similarly verify that $\#\mc E_{1,1}\leq \#\mc E_1$, $\mc B_{1,1}\subset\mc B_1$, $\beta\cap\mc N_{1,1}\subsetneq \beta\cap \mc N_1$, and $\#\beta\cap\mc N_{1,1}<\infty$. 
	
	\medskip
    In any case, we can construct the objects in (1)-(6) and (a)-(b) with subscripts modified by $_{1,1}$. 
    We also mention that $\Gamma_{1,1}^e$ is a part of $\Gamma_1$. 
    
    Suppose we have constructed the objects with subscript $_{1,i}$ for some $i\geq 1$. 
	As long as $\alpha_{1,i}\neq [0]^G$ and $\mc N_{1,i}\neq \emptyset$, we can repeat the above constructions for an element $\Gamma$ in $\mc N_{1,i}$ by using $\Gamma_{1,i}^e$ in place of $\Gamma_1$, which gives the objects in (1)-(6) and (a)-(b) with subscripts $_{1,i+1}$. 
	Moreover, in each step, if $\Gamma\in \mc N_{1,i}$ is $(G,1)$-sided, then $\#\mc B_{1,i+1}$ will strictly decrease by at least $1$ due to \eqref{Eq: decrease B_1}. 
	If $\Gamma\in \mc N_{1,i}$ is $(G,2)$-sided, then for $\beta=[\Gamma]^G\in \mc B_{1,i}$, we see from the constructions in {\bf Case 2-4} that $\beta\cap\mc N_{1,i+1}\subsetneq \beta\cap \mc N_{1,i}$ and $\#\beta\cap\mc N_{1,i+1}<\infty$, 
    which implies that $\beta$ will be eliminated from $\mc B_{1,i}$ in finite steps. 
    Note that $\mc H^G(M;\mb Z_2)$ is finite (Lemma \ref{Lem: finite G-homology class}) and there are only finitely many $(G,1)$-sided minimal $G$-hypersurfaces in $M$ by the assumption. 
    Hence, combined with $\mc B_{1,j}\subset\mc B_{1,i}$ for $j>i$, this procedure must stop in finitely many (say $m_1$) steps so that either $\alpha_{1,m_1}= [0]^G$ or $\mc N_{1,m_1}=\emptyset$. 

    \medskip
    \noindent{\bf Step 4.b.}
	If $\alpha_{1,m_1}\neq [0]^G$ and $\mc N_{1,m_1}=\emptyset$. 
    Then by the constructions, $\emptyset\neq\Gamma_{1,m_1}^e\subset \Gamma_1$, every element in $\mc E_{1,m_1}\neq\emptyset$ has an expanding $G$-neighborhood, and $\#\mc E_{1,m_1} \leq\#\mc E_1$. 
	Hence, we can pick a new element $\Gamma_{2}\neq \Gamma_{1,m_1}^e\in \mc E_{1,m_1}$ (if it exists) and apply the constructions in {\bf Step 4.a} to $\Gamma_2$ in $N_{1,m_1}$. 
	By the same procedure, we obtain new objects with subscript $_{2,m_2}$ corresponding to those in (1)-(6) and (a)-(b) 
	so that either $\alpha_{2,m_2}=[0]^G$ or $\mc N_{2,m_2}=\emptyset$. 
	We can repeat for a finite number ($\leq\#\mc E_{1}$) of times until we obtain objects with subscripts $_{k,m_k}$ so that either
	\begin{itemize}
		\item $\alpha_{k,m_k}=[0]^G$, or
		\item for any $\Sigma\in\mc E_{k,m_k}\neq \emptyset$, $\Sigma$ has an expanding $G$-neighborhood, and there is no minimal $G$-hypersurface $\Gamma\subset \interior(N_{k,m_k})$ with a non-expanding $G$-neighborhood satisfying $\Sigma\cap\Gamma=\emptyset$. 
	\end{itemize}
	For the case that $\alpha_{k,m_k}=[0]^G$, Theorem \ref{Thm: infinite in trivial G-homology class} gives a contradiction as in {\bf Step 1}.  
	Suppose next that $\alpha_{k,m_k}\neq [0]^G$, and we can make the following claim. 
	
	\begin{claim}\label{Claim: G-Frankel}
		If $\alpha_{k,m_k}\neq[0]^G$, then for any $\Sigma\in\mc E_{k,m_k}$ and minimal $G$-hypersurface $\Gamma\subset \interior(N_{k,m_k})$, 
		\begin{itemize}
			\item[(i)] $\Sigma\cap\Gamma \neq \emptyset$;
			\item[(ii)] $\Sigma$ is $G$-connected;
			\item[(iii)] $\Sigma$ is $G$-separating in $N_{k,m_k}$ provided that it is $(G,2)$-sided. 
		\end{itemize}
	\end{claim}
	\begin{proof}
		Suppose that there exists a $G$-connected minimal $G$-hypersurface $\Gamma\subset \interior(N_{k,m_k})$ which dose not intersect with $\Sigma\in\mc E_{k,m_k}$. 
		Then, by the construction of $N_{k,m_k}$, $\Gamma\cup\Sigma$ has an expanding neighborhood, and we can cut $N_{k,m_k}$ along $\Gamma\cup\Sigma$ to have a new $G$-manifold $N'$ so that $\bd N'$ has two new $G$-components $\Gamma',\Sigma'$ coming from $\Gamma$ and a $G$-component of $\Sigma$ respectively. 
		 By minimizing the area in the $G$-homology class of $\Sigma'$ in $N'$, we obtain a minimal $G$-hypersurface $S\subset N'$ with a contracting $G$-neighborhood.
		 In particular, since $\Gamma'$ and $\Sigma'$ have expanding $G$-neighborhoods, $S$ has $G$-components in $\interior(N')\subset \interior(N_{k,m_k})$ that are disjoint from $\Sigma$ but have a contracting $G$-neighborhood. 
		 This contradicts the assumption on $N_{k,m_k}$. 
		 A similar argument also shows the desired results (ii) and (iii) of $\Sigma$. 
		 %
	\end{proof}
	
	After modifying the sub-scripts $_{k,m_k}$ to $_2$, we move on to the final {\bf Step 5}.
	
	\medskip
	\noindent{\bf Step 5.}
	In summary, we have the following objects:
	\begin{itemize}
		\item[(1')] a compact $G$-manifold $N_2$ so that $\bd N_2$ is minimal with a contracting $G$-neighborhood, and $\interior(N_2)$ is $G$-equivariantly isometric to a $G$-domain $M_2$ in $M$; 
		\item[(2')] $\Gamma_{2}^e\subset \interior(N_{1})$ is a non-trivial minimal $G$-hypersurface of locally $G$-boundary-type with an expanding $G$-neighborhood; 
		\item[(3')] $\alpha_{2}:=[ \Gamma_{2}^e]^G\neq [0]^G \in \mc H^G(M;\mb Z_2)$, and 
        $\wti\alpha_{2}:=[ \Gamma_{2}^e]_{rel}+\mc B^G_{rel}(N_2,\bd N_2;\mb Z_2)$;
		\item[(4')] $\Sigma_{2}$ is an embedded minimal $G$-hypersurface in $M$ with $\alpha = \alpha_{2} + [\Sigma_{2}]^G$ and $\Sigma_{2}\cap M_2=\emptyset$;
		\item[(5')] $\mc E_{2}$ is the finite set of closed minimal realizations of $\wti\alpha_{2}$ supported in $\interior(N_2)$ so that every element in $\mc E_2$ is $G$-connected with an expanding $G$-neighborhood and intersects with any other minimal $G$-hypersurface in $\interior(N_2)$ (by Claim \ref{Claim: G-Frankel}). 
        \item[(6')] $\Sigma'\subset \bd M_2$ is a closed minimal $G$-hypersurface associated to $\Sigma\in \mc E_{2}$ so that $\Sigma'$ is isometric to a part of $\bd N_{2}$, and $\Sigma \sqcup (\Sigma_{2}\triangle \Sigma')$ is a minimal realization of $\alpha$.
	\end{itemize}
	
	If $\bd N_2=\emptyset$, i.e. $N_2=M$ and $\wti\alpha_2=\alpha\neq [0]^G$. 
    Then every min-max $G$-width $\omega_p(\alpha)$ is realized by the area of a disjoint union of $G$-connected minimal $G$-hypersurfaces $\wti\Sigma_p=\cup_{i=1}^{l_p} \wti\Sigma_p^{(i)}$ with integer multiplicity $m_p^{(i)}\in\mb N$ on each $\wti\Sigma_p^{(i)}$. 
    By Proposition \ref{Prop: lift group action}, one can apply the proof of \cite{lwy24}*{Proposition 5.4} to show that $\wti\Sigma_p$ contains a minimal realization of $\alpha$ (see also Theorem \ref{Thm: confined min-max hypersurface}(ii)).
    By the Frankel-type property in $\interior (N_2)=M$ (Claim \ref{Claim: G-Frankel}), we conclude that $l_p=1$, and $\wti\Sigma_p$ is a $G$-connected minimal realization of $\alpha$ so that $\omega_p(\alpha)=m_p\Area(\wti\Sigma_p)$.  
	Combined with the Weyl Law for equivariant volume spectrum (\cite{wang2025density}*{Theorem 1.5} or \eqref{Eq: Weyl law}), the proof of \cite{marques2017existence}*{Theorem 6.1 and 1.1} can be taken almost verbatim with $\alpha$ in place of $\mc Z_n(M;\mb Z_2)$ to show that there are infinitely many $G$-connected minimal realizations of $\alpha$. 
	This gives a contradiction. 
	
	Next, suppose that $\bd N_2\neq\emptyset$. If there exists $\Sigma\in \mc E_2$ that is $(G,2)$-sided, then by Claim \ref{Claim: G-Frankel}, $\Sigma$ is $G$-connected and $G$-separating in $N_2$. 
    Thus, there exists an open $G$-set $\Omega\subset \interior(N_2)$ so that $\bd\Omega\cap \interior(N_2)=\Sigma$. 
    After identifying $\Omega$ as a $G$-domain in $M_2\subset M$ and taking the closure $\closure(\Omega)$ in $M$, we have a minimal $G$-hypersurface $\Sigma''\subset \bd \closure(M_2)$ with $\Sigma\sqcup\Sigma''= \bd\closure(\Omega)$ so that $\Sigma''$ is isometric to the union of some $G$-components of $\bd N_2$. 
    In particular, we have $[\Sigma]^G=[\Sigma'']^G\in \mc H^G(M;\mb Z_2)$. 
    Combined with $\alpha=[\Sigma]^G+[\Sigma_2\triangle \Sigma']^G$ in (6'), we know $\alpha=[\Sigma'']^G+[\Sigma_2\triangle\Sigma']^G$, and $\Sigma_2':=\Sigma''\triangle(\Sigma_2\triangle\Sigma')$ is a minimal realization of $\alpha$ contained in $ M\setminus M_2$. 
    However, by Theorem \ref{Thm: infinite in trivial G-homology class}, there are infinitely many distinct minimal realizations $\Gamma$ of $[0]^G\in \mc H^G(M;\mb Z_2)$ that are supported in $\closure(M_2)$, which gives infinitely many distinct minimal realizations $\Sigma_2'\triangle\Gamma$ of $\alpha$ as a contradiction. 
    
    Now we know that every $\Sigma\in \mc E_2$ is $(G,1)$-sided.
    Let $\{\widehat\Sigma_{i}\}_{i=1}^m$ be the $G$-components of $\bd N_2$. 
	Then by the proof of \cite{song2023infinite}*{Lemma 13}, every minimal $G$-hypersurface $\Sigma\in\mc E_2$ satisfies that 
	\begin{itemize}
		\item $2\Area(\Sigma) > \sup_{i\in\{1,\dots,m\}} \Area(\widehat\Sigma_i)$ (as $\Sigma$ is $(G,1)$-sided). 
	\end{itemize}
	Consider the non-compact $G$-manifold $Cyl(N_2)$ with cylindrical ends associated to $N_2$, and the $G$-homology class $[\Gamma_2^e]^G_c$ with compact support. By the above reasoning, $\Gamma^e_2$ is not $G$-separating in $N_2$.
	It follows from the Frankel-type property in $\interior(N_2)$ (Claim \ref{Claim: G-Frankel}) and Theorem \ref{Thm: confined min-max hypersurface}(ii) that $\omega_p([\Gamma_2^e]^G_c)$ is realized by the area of a $G$-connected minimal realization $\wti\Sigma_p\in\mc E_2$ of $\wti\alpha_2$ compactly supported in $\interior(N_2)$ with odd integer multiplicity $m_p=2k_p+1\in\mb N$, i.e. 
	\[ \omega_p([\Gamma_2^e]^G_c) = m_p \Area(\wti\Sigma_p), \]
	for any $p\in\mb N$. 
	Together with the finiteness of $\mc E_2$ and Theorem \ref{Thm: width estimates in cylindrical mfd}, we have a contradiction to the arithmetic result \cite{song2023infinite}*{Lemma 14}. 
\end{proof}

\bibliographystyle{abbrv}

\bibliography{reference.bib}

\end{document}